\documentclass[12pt,oneside]{amsart}

\usepackage{enumerate}
\usepackage{graphicx}%
\usepackage{multirow}%
\usepackage{amsmath,amssymb,amsfonts}%
\usepackage{amsthm}%
\usepackage{mathrsfs}%
\usepackage[title]{appendix}%
\usepackage{xcolor}%
\usepackage{color}
\usepackage{textcomp}%
\usepackage{manyfoot}%
\usepackage{booktabs}%
\usepackage{algorithm}%
\usepackage{algorithmicx}%
\usepackage{algpseudocode}%
\usepackage{listings}%

\newcommand{\rid}[1]{E(\mathcal{#1})}
\newcommand{\ridx}[2]{E(\mathcal{#1}(#2))}

\newcommand{\invcat}{\mathrm{InvCat}}
\newcommand{\K}{\mathbb{K}}
\newcommand{\iicp}{\mathcal{I}_{ic}(P)}
\newcommand{\ic}[1]{\mathcal{I}_{ic}(#1)}
\newcommand{\pex}[2]{P_{#1,#2}(\mathcal{C})}
\newcommand{\inc}{\widehat{inc}}
\newcommand{\calC}{\mathcal{C}}
\newcommand{\calD}{\mathcal{D}}
\newcommand{\pc}{ P_{\mathcal{C}}}

\newtheorem{theorem}{Theorem}[section]
\newtheorem{lemma}[theorem]{Lemma}
\newtheorem{corollary}[theorem]{Corollary}

\theoremstyle{definition}
\newtheorem{definition}[theorem]{Definition}
\newtheorem{example}[theorem]{Example}
\newtheorem{proposition}[theorem]{Proposition}

\theoremstyle{remark}
\newtheorem{remark}[theorem]{Remark}

\numberwithin{equation}{section}

\newcommand{\zero}{\mathbf{0}}

\begin{document}

\title{Partial actions of inverse categories and their algebras}


\author{Marcelo M. Alves}

\curraddr{Mathematics Department, Federal University of Paran\'a (UFPR), Brazil}
\email{marcelomsa@ufpr.br}

\author{Willian G.G. Velasco}\curraddr{Mathematics Department, Federal University of Technology – Paran\'a (UTFPR), Brazil}
\email{willianvelasco@protonmail.com}



\begin{abstract}
    In this work we introduce partial and global actions of inverse categories on posets in two variants, fibred actions and actions by symmetries. We study in detail actions of an inverse category $\mathcal{C}$ on specific subposets of the poset of finite subsets of $\mathcal{C}$, the \textit{Bernoulli actions}.   
We show that to each fibred action of an inverse category on a poset there corresponds another inverse category, the semidirect product associated to the action. The Bernoulli actions give rise to the \textit{Szendrei expansions} of $\mathcal{C}$, which define an endofunctor of the category of inverse categories. We extend the concept of ``enlargement'' from inverse semigroup theory to inverse categories, and we show that if $\mathcal{D}$ is an enlargement of $\mathcal{C}$ then their Cauchy completions are equivalent categories; in particular, some pairs corresponding to partial and global Bernoulli actions are enlargements.  
We conclude by studying convolution algebras of finite inverse categories and showing that if $\mathcal{D}$ is an enlargement of $\mathcal{C}$ then their convolution algebras are Morita equivalent. Furthermore, using Kan extensions we also analyze the infinite case.
\end{abstract}



\maketitle

\section{Introduction}

Partial group actions were introduced in the 1990's by R. Exel, in the context of $C^*$-algebras, as a means to generalize the construction of the crossed product associated to a group action, while simultaneously extending homological results for standard crossed products \cite{exel1994circle}. These so-called partial crossed products were later investigated in a purely algebraic context, beginning by Dokuchaev and Exel's work on associativity of crossed products in \cite{ruy05}.

This foundational work has since evolved into several specialized areas, notably for \textit{inverse semigroups} and \textit{groupoids}. The main goal of this paper is to unify and generalize these parallel approaches within the single, comprehensive framework of \textit{inverse categories}. This approach is motivated by the fact that inverse categories inherently combine fundamental characteristics of both inverse semigroups and groupoids, and it follows a successful pattern of extending other key results to the categorical setting, such as the Wagner-Preston Theorem (cf. \cite{linckelmann-inverse}) and the Ehresmann-Schein-Nambooripad Theorem (cf. \cite{dewolf-ehresmann}).

Our unification builds directly on the foundational work in these related areas. For \textit{inverse semigroups}, a fundamental result by Exel is that partial actions of a group $G$ correspond to actions of an inverse semigroup $S(G)$ \cite{ruy98}, a construction later connected to the Birget-Rhodes expansion \cite{birget1984almost,kellendonk-lawson}. For \textit{representation theory and groupoids}, the partial group algebra $\K S(G)$ was shown to be isomorphic to a groupoid algebra $\K \Gamma (G)$ \cite{mishca-ruy-paolo}, itself associated with a ``Bernoulli action'' \cite{ruylivro}. We mention that other investigations have already explored expansions of groupoids \cite{gilbert-pdgexp,gilbert-pthm,bagio2010partial} and inverse semigroups \cite{lawsonmarstein-expansions,bussexel-expa}.

Our framework is built upon Bernoulli actions. In Section \ref{section-actions} we introduce the core concept of a fibred action of an inverse category on a poset, and we establish two partially ordered sets (posets) derived intrinsically from an underlying structure—whether an inverse category, groupoid, inverse semigroup or group. When $\mathcal{C}$ is an inverse category we denote these posets by $\mathcal{P}(\mathcal{C})$ and $\mathcal{P}_\circ(\mathcal{C})$, and the corresponding fibred actions are \emph{Bernoulli actions}.  

In Section \ref{section-sd-product} we construct a general semidirect product associated to a fibred action.
Specializing this construction to Bernoulli actions we obtain the \textit{Szendrei expansions} of an inverse category. 
Our work harmonizes several well-known expansion mechanisms. Inspired by Gilbert’s expansions for ordered groupoids \cite{gilbert-graphexp} and the prefix expansions developed by Exel \cite{bussexel-expa}, Buss, and Lawson–Margolis–Steinberg \cite{lawsonmarstein-expansions}, we employ Linckelmann’s \cite{linckelmann-inverse} algebraic framework as a unifying pathway to treat these constructions under a single categorical approach.

In our joint study of partial and global actions on posets, we draw key inspiration from Lawson’s concept of \emph{enlargement} \cite{lawson-enlarg} (generalized by Steinberg \cite{steinberg-morita}) as well as the techniques introduced in Elizabeth Caroline Miller’s thesis \cite{miller-thesis-structure} (advised by N.D. Gilbert). We extend 
the concept of enlargement to inverse categories. 
Our approach captures both internal domain/range properties and order-theoretic structures defined by the inverse category. 
We show that certain Szendrei expansions are enlargements (Proposition \ref{propo-invcat-enlarg-implies-cauchyenlarg}, Theorem \ref{theo-restrict-exp-enlarg}),  and we also show that if $\mathcal{D}$ is an enlargement of $\mathcal{C}$ then their Cauchy completions are equivalent (Proposition \ref{propo-invcat-enlarg-implies-cauchyenlarg}). 

Section \ref{section-algebras} explores the consequences for representation theory, proving that if $\mathcal{D}$ is an enlargement of $\mathcal{C}$ and both are finite inverse categories then their convolution algebras $\K \mathcal{D}$ and $\K \mathcal{C}$ are Morita equivalent (Theorem \ref{theo-enlarg-imp-morita-alg}, Corollary \ref{coro-szendrei-alge-morita}); we also analyze the infinite case (Theorem \ref{theo-rep-inv-cat-exp}). In the specific setting of groups, this construction recovers the partial group algebra $\K_{par}(G)$ and introduce $\K_{glob}(G)$ as a novel construction, representing one of the core original contributions of this paper.

Section \ref{section-previous-cases} confirms our framework's general-purpose nature by relating our constructions back to the known expansions in the literature 
\cite{ruy98,bussexel-expa,lawsonmarstein-expansions,gilbert-pdgexp}.

Finally, we highlight that our order-theoretic viewpoint complements recent studies on set-theoretic partial semigroupoid actions on sets, providing a deeper understanding of inverse categories, expansion theory, and partial actions (cf. \cite{augusto2026ap,petasny2026partial}).

\section{Preliminary notions }\label{section-preliminary}

\subsection{Inverse categories} Inverse categories are the main ingredient of our constructions, and we review here basic definitions and results on this subject. Our approach is closer to the theory of inverse semigroups as in \cite{linckelmann-inverse}; other way of dealing with this theory is via restriction categories, as it is done in Cockett-Lack \cite{cortes-partialactoncat}.

\begin{definition}[\cite{linckelmann-inverse}]\label{defi-inv-cat}
	A category $\mathcal{C}$ is called an \textit{inverse category} if for each arrow $(s:X\rightarrow Y)\in\mathcal{C}$ there exists a unique arrow $(s^\circ:Y\rightarrow X)\in\mathcal{C}$ such that $ss^\circ s=s \text{  and  } s^\circ ss^\circ=s^\circ.$ Notation: $(\mathcal{C}, (\:\:)^\circ)$ will denote an inverse category.
\end{definition}

As in the case of inverse semigroups, there is an equivalent definition on inverse category via the commutation of the idempotent morphisms $s^\circ s$. 

\begin{proposition}[\cite{cockett-restriction-I}]\label{prop-defi-equiv-invcat}
	Let $\mathcal{C}$ be a category. The following are equivalent:
	\begin{enumerate}[(i)]
		\item for every arrow $s:X\rightarrow Y$ there exists a unique arrow $t:Y\rightarrow X$ such that $sts=s$ and $tst=t$;
		\item there exists a functor $(\:\:)^\circ:\mathcal{C}\rightarrow\mathcal{C}^{op}$ which fixes objects ( $ (X)^\circ=X$ for every $X\in\mathcal{C}^{(0)}$) and for $s,t\in\mathcal{C}$, satisfies the following equalities: 
  \begin{center}
$(s^\circ)^\circ=s$,  $ss^\circ s=s$, and $(ss^\circ)(tt^\circ)=(tt^\circ)(ss^\circ).$
\end{center}	
		Moreover, the structures described in (i) and (ii) are unique.
	\end{enumerate}
\end{proposition}

Functors between inverse categories preserve the inverse structure.

\begin{proposition} [\cite{nystedt-pinedo-epsilon}]\label{prop.functors.inverse.categories}
  Let   $(\mathcal{C}, (\:\:)^\circ)$ and $(\mathcal{D}, (\:\:)^\ast)$ be inverse categories and let $F: \mathcal{C} \to \mathcal{D}$ be a functor. Then $F(s^\circ) = F(s)^\ast$ for every morphism $s \in \mathcal{C}$. 
\end{proposition}

\begin{proposition}[\cite{linckelmann-inverse}]\label{prop-inv-cat-props}
	Let $(\mathcal{C}, (\:\:)^\circ)$ be an inverse category. 
	Given arrows $s,t:X\rightarrow Y$, and the idempotent morphisms $e:X\rightarrow X, f:Y\rightarrow Y$ in $\mathcal{C}$, we have that:
	\begin{enumerate}
		\item [(i)]$(st)^\circ =t^\circ s^\circ $, and
		\item [(ii)]the arrows $ses^\circ:X\rightarrow X$  and $tft^\circ: Y\rightarrow Y$ are idempotent.
	\end{enumerate}
\end{proposition}

For any pair of arrows $s,t:X\rightarrow Y$ in an inverse category $\mathcal{C}$, we write $s\leqslant t \text{ if } s=te \text{ for some idempotent } e:X\rightarrow X$. Analogously to the natural order of inverse semigroups \cite{lawsonlivro}, this relation is a partial order on the inverse category $\mathcal{C}$ \cite{linckelmann-inverse} and it has the following equivalent characterizations.

\begin{proposition}[\cite{linckelmann-inverse}]\label{prop-inv-cat-ord}
	Let $(\mathcal{C}, (\:\:)^\circ)$ be an inverse category and let $s,t:X\rightarrow Y$ and $f:Y\rightarrow Y$ be arrows. The following are equivalent:
	\begin{enumerate}[(i)]
		\item $s\leqslant t$;
		\item $s=ft$ for some idempotent morphism $f$;
		\item $s=ss^\circ t$;
		\item $s=ts^\circ s$.
	\end{enumerate}
	Moreover, if $p,q:Y\rightarrow Z$ is another pair of arrows, such that $p\leqslant q$ and exists $ps$ and $qt$, then $ps\leqslant qt$.
\end{proposition}

Idempotent morphisms play a fundamental role in our definition of fibred actions of an inverse category (Definition \ref{defi-invcat-action}). 
The following definition will facilitate the identification of idempotents related to a given arrow.
\begin{definition}\label{defi-Rid-inner-outer-source-target}
	Let $\mathcal{C}$ be an inverse category and $s: X \rightarrow Y$ be a morphism of $\mathcal{C}$.
	\begin{enumerate}[(I)]
		\item The \textit{set of idempotent morphisms} will be denoted by $\rid{C}$. Furthermore, if $X\in\mathcal{C}^{(0)}$, the set of idempotent morphisms in $X$ will be denoted by $\ridx{C}{X}$.
		\item The \textit{outer domain} (or \textit{outer source}) and the \textit{outer range} (or \textit{outer target}) of $s$ are the maps  $od, or: \mathcal{C}\rightarrow \mathcal{C}^{(0)}$ with $od(s)=X \text{ and } or(s)=Y.$
		\item The \textit{inner domain} (or \textit{inner source}) and the \textit{inner range} (or \textit{inner target}) of $s$ are  $id, ir: \mathcal{C}\rightarrow \rid{C}$ with $id(s)=s^\circ s \text{ and } ir(s)=ss^\circ.$
	\end{enumerate}
	
	Combining (II) and (III) we will denote the domain (or source) and the range (or target) pairs maps by: $d=(od,id): \calC \to \calC^{0} \times E(\calC) \text{  and  } r = (or,ir): \calC \to \calC^{0} \times E(\calC) .$  
	
\end{definition}

Borrowing the idea of the Green classes of inverse semigroups \cite{lawsonlivro} and the set of arrow beginning, or terminating, in a particular object \cite{gilbert-pdgexp}, we have the next definition.

\begin{definition}\label{defi-costars-rsets-inv-cat}
	Let $(\mathcal{C}, (\:\:)^\circ)$ be an inverse category and let $s$ and $t$ be arrows in $\mathcal{C}$.
	\begin{enumerate}	
		\item[(I)] We say that $s,t\in\mathcal{C}$ are $\mathscr{L}$\textit{-related }if $id(s)=id(t)$, or simply $s^\circ s=t^\circ t$. The set of  arrows $\mathscr{L}$-related to $s$ will be denoted by $\mathscr{L}_s$.
		\item[(II)] We say that $s,t\in\mathcal{C}$ are $\mathscr{R}$\textit{-related }if $ir(s)=ir(t)$, in  other terms $ss^\circ=tt^\circ$. The set of  arrows $\mathscr{R}$-related to $s$ will be denoted by $\mathscr{R}_s$.
		
	\end{enumerate}
	
	Let $X$ and $Y$ be objects in $\mathcal{C}^{(0)}$.
	\begin{enumerate}
		\item [(III)] The set of all arrows of $\mathcal{C}$ starting in $X$ is $\text{Star}(X):=\{t\in\mathcal{C}; od(t)=X  \}$.
		\item [(IV)] The set of all arrows of $\mathcal{C}$ ending in $Y$ is $\text{Costar}(Y):=\{s\in\mathcal{C}; or(s)=Y  \}$.
		
	\end{enumerate}

\end{definition}


\subsection{Cauchy completion}

Let $\mathcal{C}$ be a small category; there is always an auxiliary category, constructed upon the idempotent arrows of $\mathcal{C}$, such that $\mathcal{C}$ may be embedded in this category. This construction is the Cauchy completion of $\mathcal{C}$ (also called idempotent completion, or Karoubian completion) and in the following  we will present its main aspects. In particular, using the Cauchy completion of an inverse category we can define a groupoid which is analogous to the restriction groupoid of an inverse semigroup \cite{lawsonlivro}.

\begin{definition}[\cite{borceux-handbook-I}]\label{defi-spli-idp}
	Suppose $X\in\mathcal{C}^{(0)}$ and let $e:X\rightarrow X$ be an idempotent, we say that the \textit{idempotent  $e$ splits} if there are an object $Y\in\mathcal{C}^{(0)}$ and arrows $s:X\rightarrow Y$, $t:Y\rightarrow X$, such that $e=ts \text{  and  } st=1_Y.$
	
	When all idempotents split, we say that $\mathcal{C}$ is a \textit{Cauchy complete} category.
\end{definition}

Let $\mathcal{C}$ be a category, its Cauchy completion $\widehat{\mathcal{C}}$ can be constructed as follows (see for instance \cite{linckelmann-inverse}): 
\begin{itemize}
	\item \underline{objects}: an object of $\widehat{\mathcal{C}}^{(0)}$ is a pair $(X,e)$, where $ X\in\mathcal{C}$ and $e^2=e\in\mathcal{C}(X,X)$;
	\item \underline{arrows}: a morphism is a triple $(e,s,f):(X,e)\rightarrow (Y,f)$, where $s:X\rightarrow Y$ is an arrow of $\mathcal{C}$ satisfying: $se=s=fs$;
	\item \underline{composition}: the composition of $(e,s,f): (X,e)\rightarrow (Y,f)$ and $(f,t,g):(Y,f)\rightarrow (Z,g)$ is the arrow  $(f, t, g)(e,s,f)  =(e,ts,g):(X,e)\rightarrow (Z,g)$. 
\end{itemize}

In the particular case when  $(\mathcal{C}, (\:\:)^\circ)$ is an inverse category,  $(X,e)$ is isomorphic to $(Y,f)$ if, and only if, there is an arrow $s:X\rightarrow Y$ such that $s^\circ s=e \text{  and  } ss^\circ=f.$

\begin{proposition}[\cite{linckelmann-inverse}]\label{prop-equiv-cauchy-inv}
	Let $\mathcal{C}$ be a small category. The category $\mathcal{C}$ is an inverse category if, and only if, $\widehat{\mathcal{C}}$ is an inverse category.
\end{proposition}

Our final considerations about Cauchy completions, for now, are three properties compiled in the next proposition, taken from Borceux \cite{borceux-handbook-I}. 

\begin{proposition}[\cite{borceux-handbook-I}]\label{propo-cauchy-compl-properties}
	Let $\mathcal{C}$ be a small category, its Cauchy completion $\widehat{\mathcal{C}}$ satisfies:
	\begin{enumerate}[(i)]
		\item $\widehat{\mathcal{C}}$ is small;
		\item $\mathcal{C}$ is a full subcategory of $\widehat{\mathcal{C}}$, where the inclusion functor is defined by $X\in\mathcal{C}^{(0)}\mapsto (X,1_X)\in \widehat{\mathcal{C}}^{(0)}$ and $(s:X\rightarrow Y)\in\mathcal{C} \mapsto (1_X,s,1_Y)\in\widehat{\mathcal{C}}$ ;
		\item the inclusion of $\mathcal{C}$ in $\widehat{\mathcal{C}}$, which sends $X\in\mathcal{C}^{(0)}$ to $(X,1_X)\in\widehat{\mathcal{C}}^{(0)}$ is an equivalence if, and only if, every idempotent in $\mathcal{C}$ splits.
	\end{enumerate}
\end{proposition}

We will return to Cauchy completions further ahead, when we study algebras of Szendrei expansions. Moreover, we will use a groupoid associated to an inverse category.

\begin{definition}\label{defi-restriction-gpd}[\cite{linckelmann-inverse}]
	The \textit{restriction groupoid} associated to a small category $\mathcal{C}$, by notation $\mathcal{G}_{\mathcal{C}}$, is the subcategory of $\widehat{\mathcal{C}}$ defined by
	\begin{itemize}
		\item \underline{objects}: $\mathcal{G}_{\mathcal{C}}^{(0)}= \widehat{\mathcal{C}}^{(0)}$;
		\item \underline{arrows}: $\mathcal{G}_{\mathcal{C}}=\{  x\in\widehat{\mathcal{C}}; x \text{ is an isomorphism} \}$;
	\end{itemize}
	
\end{definition}

When $\mathcal{C}$ is an inverse category, to each  $e^2=e\in\mathcal{C}(X,X)$ there is an associated group  $$\mathcal{C}_e:=\{s\in\mathcal{C}(X,X); X\in\mathcal{C}^{(0)}, \: ss^\circ=e=s^\circ s \}.$$ It turns out that this group is isomorphic to the automorphism group $\mathcal{G}_{\mathcal{C}}((X,e), (X,e) )$ via $s\in\mathcal{C}_e\mapsto (e,s,e)\in\mathcal{G}_{\mathcal{C}}((X,e), (X,e) ).$

\subsection{The convolution algebra of a finite inverse category}

\begin{definition}[\cite{xu-representations}]\label{defi-conv-alg}
	Let $\mathcal{C}$ be a small category and $\mathbb{K}$ be a commutative ring. The \textit{category algebra}, or \textit{convolution algebra}, $\mathbb{K}\mathcal{C}$ is a free $\mathbb{K}$-module whose basis is the arrow set $\mathcal{C}$. The product  is defined on the basis by $s t=  {s \circ t}$, if  this composition exists, and 
	$s t=0$ if not.
\end{definition}

The convolution algebra is always an algebra with local units; if $\mathcal{C}^{(0)}$ is finite then the convolution algebra is unital, with unit given by $1_{\mathbb{K}\mathcal{C} }=\displaystyle\sum_{e\in\mathcal{C}^{(0)} }1_e$.

In \cite{steinberg-mobius} it was shown that the convolution algebra of a finite inverse semigroup is isomorphic to the convolution algebra of its restriction groupoid. As proved in \cite{linckelmann-inverse}, 
the same holds for finite inverse categories.

\begin{theorem}[\cite{linckelmann-inverse}] \label{theo-Linkcelmann-iso}
	Let $(\mathcal{C}, (\:\:)^\circ)$ be a finite inverse category and $\mathbb{K}$ be a commutative ring.   The convolution algebra of $\mathcal{C}$, $\mathbb{K}\mathcal{C}$, and the algebra of the restriction groupoid, $\mathbb{K}\mathcal{G}_{\mathcal{C}}$, are isomorphic. 
\end{theorem}

This result allows us to write $\mathbb{K}\mathcal{C}$ as a groupoid algebra, and groupoid algebras can be realized as direct products of matrix algebras with coefficients in group algebras. Those groups correspond to automorphism groups / isotropy groups of objects of the groupoid. We have the following corollary.

\begin{corollary}[\cite{linckelmann-inverse}] \label{coro-rest-gpd-cat-inv-alg}
	Let $(\mathcal{C}, (\:\:)^\circ)$ be a finite inverse category and $\mathbb{K}$ be a commutative ring. Let $\mathcal{E}$ be the set of representatives of the isomorphism classes of idempotent endomorphisms in $\mathcal{C}$; denote by $n_e$ the number of idempotents in $\mathcal{C}$ which are isomorphic to $e$. Then the convolution algebra of $\mathcal{C}$ has the representation $\mathbb{K}\mathcal{C}\simeq \displaystyle\bigoplus_{e\in\mathcal{E}}\mathbb{M}_{n_e}(\mathbb{K}\mathcal{C}_e).$
\end{corollary}

\begin{remark} \label{remark:Linckelman-iso-Steinberg}
    We note that Theorem \ref{theo-Linkcelmann-iso} is a consequence of \cite[Thm 4.2]{steinberg-mobius}. 
Given a finite inverse category  $\calC$, let $S(\calC)$ be the (inverse) semigroup obtained from $\calC$ by adding a zero element $\zero$, and defining the operation by extending the composition of $\calC$ by zero, that is, $st = s \circ t$ when the composition is defined, and otherwise $st = \zero$. By construction, the natural order of $S(\calC)$ restricted to $\calC$ is the natural order of the latter.

Now consider the semigroup algebra $\mathbb{K} S(\calC)$.  The $\mathbb{K}$-submodule $\mathbb{K} \zero$ is an ideal of this algebra, and 
the contracted semigroup algebra $\mathbb{K}_{\zero} S = \mathbb{K} S / \mathbb{K} \zero$ is isomorphic to the algebra $\mathbb{K} \calC$.  On the other hand, let $\Gamma$ be the restriction groupoid of $S(\calC)$. Then $\mathbb{K} \Gamma =  \mathbb{K} \mathcal{G}_{\mathcal{C}} \oplus \mathbb{K} \zero$ as a sum of ideals, since $\zero$ is a component of the groupoid $\Gamma$, and it follows that  $\mathbb{K} \Gamma / \mathbb{K} \zero \simeq   \mathbb{K} \mathcal{G}_{\mathcal{C}}$ as algebras. Now  \cite[Thm 4.2]{steinberg-mobius} presents an isomorphism  $\varphi : \mathbb{K} S(\calC) \to \mathbb{K} \Gamma$, whose inverse depends on the M\"obius function of the poset  $S(\calC)$. By computing the M\"obius function of $S(\calC)$ in terms of the M\"obius function of $\calC$ one can verify that $\varphi$ induces an isomorphism of algebras $\mathbb{K} \calC \simeq \mathbb{K}  \mathcal{G}_{\mathcal{C}}$.

\end{remark}

\section{Actions of inverse categories}\label{section-actions}
We begin with a definition of action of an inverse category on a  poset which extends both the concepts of fibred category actions on sets \cite{maclane-sheaves} and of fibred groupoid actions on posets introduced by Miller \cite{miller-thesis-structure}. Our definition considers the outer and inner structures of an inverse category.

\subsection{Actions of inverse categories on posets}

\begin{definition}\label{defi-invcat-action}
	Let $(\mathcal{C}, (\:\:)^\circ)$ be an inverse category and $(P,\leqslant)$ be a partially ordered set. A \textit{fibred ordered action}, where
	\begin{itemize}  
      \item $\rho:P\rightarrow \mathcal{C}^{(0)}\times \rid{C}$, $\rho(x) =  (o\rho(x),i\rho(x))$, \emph{moment map}, is such that $od(i\rho(x)) = o\rho(x)$ and $i\rho$ is an order preserving map;
		\item $\theta: \mathcal{C} {_d\times _\rho} P\rightarrow P$ is the action map, where 
  \[
\mathcal{C} {_d\times _\rho} P=\{ (s,x)\in \mathcal{C}\times P; o\rho(x)=od(s),\: i\rho(x)\leqslant id(s) \} \ ; 
  \]  the image of the pair $(s,x)$ by $\theta$ will be denoted, as usual, by $\theta_s(x)$, and for each $s\in \mathcal{C}$ the map 
  \[
\theta_s : \{ x \in P ; (s,x) \in \mathcal{C} {_d\times _\rho} P\} \to P, \ \ \ \ x \mapsto \theta_s(x),
  \] is order-preserving. 
\item 	For every $x \in P$, $s,t \in \mathcal{C}$, 
	
	\begin{enumerate}[(I)]
		\item  $\theta_{i\rho(x)}(x)=x$;
		\item $o\rho(\theta_s(x))=or(s)$, $i\rho(\theta_s(x)) \leq ir(s)$, and $i\rho(\theta_s(x))=ir(s)$ if $id(s)=i \rho (x)$; 
		\item $\theta_s(\theta_t(x))=\theta_{st}(x)$ if $(t,x) \in \mathcal{C} {_d\times _\rho}$ and the composition $st$ in defined in $\mathcal{C}$. 
	\end{enumerate}

\end{itemize}

	Notation: $(\rho,\theta):(\mathcal{C},(\:\:)^{\circ})\curvearrowright (P,\leqslant)$ will denote a fibred ordered action.
\end{definition}

As we will deal only with ordered fibred actions, we will call  them simply \textit{fibred actions}.

\begin{example}
Let     $(\mathcal{C},(\:\:)^{\circ})$ be an inverse category and let $\pc$ be the underlying poset of $\mathcal{C}$. Then there is a canonical fibred action of  $\mathcal{C}$ on $\pc$ given by the following: $o\rho(s) = or(s)$, $i \rho(s) = ir(s)$, and $\theta (s,x) = s x$. In fact, $\theta$ is well-defined because $\mathcal{C}_d \times {}_\rho \pc \subseteq \mathcal{C}_d \times {}_r \mathcal{C}$, and the other conditions that define a fibre action are satisfied as consequences of the definition and basic properties of an inverse category.
\end{example}

\begin{example} (action on idempotents)
Let     $(\mathcal{C},(\:\:)^{\circ})$ be an inverse category and let $Q$ be the underlying poset of $\rid{C}$. The inverse category $\mathcal{C}$ acts on $Q$ by the following: $o\rho(e) = or(e)$, $i \rho(e) = e$, and $\theta (s,e) = s e s^{0}$. 
\end{example}

\begin{proposition} \label{proposition.theta_s.bijection}
Let $(\rho, \theta)$ be a fibred action of the inverse category $\mathcal{C}$ on the poset $P$.
\begin{enumerate}[(i)]
    \item If $e \in E(\mathcal{C})$ and $i \rho(x) \leq id(e) = e$ then  $\theta_e(x) = x$.
    \item Let $D_s = \{x \in P ; (s^\circ,x) \in \mathcal{C} {_d\times _\rho} P \}$. Then $\theta_s : D_{s^\circ} \to D_s$ is a bijective order-preserving map. 
\end{enumerate}
\end{proposition}
\begin{proof}
    It follows from (I) and (III) that if $e \in E(\mathcal{C})$ and $i \rho(x) \leq id(e) = e$ then  $\theta_e(x) = x$, since  $i \rho(x) \leq id(e) = e$ if and only if $e \ i\rho(x) = i\rho(x)$, and therefore $\theta_e(x) = \theta_e(\theta_{i\rho(x)}(x))= \theta_{e\ i\rho(x)}(x) =\theta_{i\rho(x)}(x) = x$.
 
    Given a fibred action $\theta$ of $\mathcal{C}$ on $P$, for $s \in \mathcal{C}$ let $D_s$ denote the domain of $\theta_s$, $$D_s = \{x \in P ; (s^\circ,x) \in \mathcal{C} {_d\times _\rho} P \}.$$ Then $\theta_s$ is a bijection from $D_{s^\circ} $ to $D_s$. In fact, $\theta_s(x)$ is defined if and only if $x \in D_{s^\circ} $ and, by Def. \ref{defi-invcat-action} (II), $\theta_s(x) \in D_{s} $. Using (III) and the previous item one concludes that $\theta_{s^\circ} \theta_s = Id_{D_{s^\circ}}$ and $\theta_s \theta_{s^\circ} = Id_{D_{s}}$.
\end{proof}

\begin{remark} 
Consider a fibred action of $\mathcal{C}$ on $P$. Then $P_X = o \rho^{-1}(X)$ is an ideal of $P$ for each $X \in \mathcal{C}^{(0)}$ and $P = \sqcup_{X \in \mathcal{C}^{(0)}}P_X $ (disjoint union). This follows straight from the definitions: 
if $p\in P_X$ and $q \in P$ is such that $p \geq q$ then $i \rho (p) \geq i \rho (q)$ and, by the definition of order in $\mathcal{C}$ and the fact that $i \rho(p)$ and $i \rho(q)$ are idempotent morphisms, 
 $
i \rho (q) = i \rho(q) (i \rho(q))^\circ i \rho (p) = i \rho (q) i \rho (p).  
 $
In particular $i \rho (p)$ and $i \rho(q)$ are composable idempotent morphisms and therefore $X = o \rho(p) = o \rho(q)$, hence $q \in P_X$.
\end{remark}

In {\cite{gilbert-pdgexp}} Gilbert investigates groupoid actions by symmetries. 
We will present another definition of action of an inverse category on a poset which includes the former as a particular case. In order to do that we must define a new inverse category associated to a given poset.  A similar construction, without the ordering, has already appeared in Linckelmann \cite{linckelmann-inverse} and Nystedt-Oinert-Pined \cite[Def. 12]{nystedt-pinedo-epsilon}.  
 Let $(P,\leqslant)$ be a poset, in this case, we will write ``$I \triangleleft P$''. For technical reasons, we include the empty set as an ideal. 
\begin{proposition} Let  $(P,\leqslant)$ be a poset. The following data determine an inverse category $\iicp$:
\begin{itemize}
	\item \underline{objects}: the objects of $\iicp$  are the ideals of $P$;
	\item \underline{arrows}: given order ideals $U,V \triangleleft P$, a morphism from $U$ to $V$ is a order preserving bijection $s:U'\rightarrow V'$, where $U' \triangleleft U$ and $V' \triangleleft V$;
	\item \underline{composition}: let $U, V, W \triangleleft P$, also $U' \triangleleft U, V'\triangleleft V, V''\triangleleft V,  W' \triangleleft W$, and let  $s:U'\rightarrow V'$ and $t: V''\rightarrow W'$ be order preserving bijections; the composition $ts$ is defined by $$ts=t_{|V'\cap V''}\circ s_{|s^{-1}(V'\cap V'') }: s^{-1}(V'\cap V'')\rightarrow t(V'\cap V'');$$ 
	\item \underline{inverses}: the inverse structure is the inversion of maps.
\end{itemize}

\end{proposition}
The proof follows from showing that $\iicp$ is an inverse subcategory of the category of partial bijections defined in \cite[Def. 12]{nystedt-pinedo-epsilon}.

\begin{remark} 
\label{remark.domain.Iic(P)}
Let $f: U' \to V' \in Hom_{\iicp}(U,V)$ be a partial bijection. Then $U' = \{x \in P ; \exists f^{-1} f(x) \textrm{ and } f^{-1} f(x)=x\}$. In fact, if $x \in U'$ then $f(x)$ is defined and $f(x) = f f^{-1}f(x)$; since $f$ is bijective we have that $x = f^{-1}f(x)$. The converse is clear. Since $f^{-1}f$ is an idempotent morphism, we conclude that the domain of $f$ coincides with the domain of $f^{-1}f$.
\end{remark}

\begin{definition}\label{defi-invcat-actions-via-symm}
	Given an inverse category $\mathcal{C}$ and a poset $(P,\leqslant)$, an \textit{inverse category action via symmetries}  (or \textit{automorphisms}) of $\mathcal{C}$ on $(P,\leqslant)$ is  a functor $\Theta: \mathcal{C}\rightarrow \iicp$ such that $P = \bigcup_{X \in \mathcal{C}^{(0)}} \Theta(X)$.
\end{definition}

Now, we can relate fibred actions to actions by symmetries. The lemma below shows how to define maps among fibers from fibred actions.

\begin{lemma}\label{lemma-invcat-action-symmetries} 
	Each fibred action of an inverse category $\mathcal{C}$ on a poset $P$ determines an action of $\mathcal{C}$ via symmetries on $P$. 
    Conversely, an action by symmetries $\Theta: \mathcal{C}\rightarrow \iicp$ determines a fibred action of $\mathcal{C}$ on $P$  if 
    \begin{enumerate}[(i)]
        \item $P$ is the \textit{disjoint union} of the sets $\Theta(X)$, $X \in \calC^{(0)}$;
        \item Every nonempty subset of $E(C)$ is has a minimum.  
    \end{enumerate}
	
\end{lemma}
\begin{proof}
	Let $(\rho,\theta): (\mathcal{C}, (\:\:)^{\circ})\curvearrowright (P,\leqslant)$ be a fibred action of the inverse category $(\mathcal{C},(\:\:)^\circ)$ on the poset $(P,\leqslant)$. It follows from the definition of fibred action and from Proposition  \ref{proposition.theta_s.bijection} that to each $s\in \mathcal{C}$ there corresponds a bijective order-preserving map 	$$\theta_s:D_{s^\circ}\rightarrow D_s, \ \  \ \theta_s(x) = \theta(s,x),$$ 
 whose domain and range are (respectively) the sets
	$D_{s^{\circ}}:=\{x\in P; o\rho(x)=od(s),\: i\rho(x)\leqslant id(s)  \}$, and $D_{s}:=\{y\in P; o\rho(y)=or(s),\: i\rho(y)\leqslant ir(s)  \}$.

  	Let $x\in P$ and $y\in D_{s^\circ}$  such that $x\leqslant y$ and $o\rho(x)=o\rho(y)$; since the map $i\rho$ is order preserving, by Definition \ref{defi-invcat-action}, we have $i\rho(x)\leqslant i\rho(y)$ and by the definition of $D_{s^\circ}$ it follows that $i\rho(x)\leqslant i\rho(y)\leqslant id(s)$, hence $D_{s^\circ}$ is an order ideal.

		If $s \in \mathcal{C}(X,Y)$ and $1_X: X\rightarrow X$ is the identity of $X$ then $id(s) = s^{\circ}s \leq 1_X$, and therefore $D_s \triangleleft D_{1_X}$; hence $\theta_s$ is a morphism $\theta_s : D_{1_X} \to D_{1_Y}$ in $\iicp$.
  Moreover, it follows that 
 $P=\bigcup_{X\in {\mathcal{C}}^{(0)}} D_{1_X}$. 
	
 Combining the previous computations we define the functor $\Theta: \mathcal{C}\rightarrow \iicp$ by $X\in\mathcal{C}^{(0)}\mapsto \Theta(X):=D_{1_X} \text{  and  } s\in\mathcal{C}\mapsto \Theta(s):=\theta_s.$ As $(\rho,\theta)$ is a fibred action, by  Definition \ref{defi-invcat-action}, items (I) and (III),  this association is indeed a functor.

 Conversely, given an action by symmetries $\Theta: \mathcal{C}\rightarrow \iicp$, 
 hypotheses (i) and (ii) define a moment map $\rho:P\rightarrow \mathcal{C}^{(0)}\times \rid{C}$, where $i\rho(x) = \min \{ e \in E(\mathcal{C}); x \in \text{domain}(\Theta(e)) \}$. Then the domain of $\Theta(s)$ is exactly the set 
\[
\{ x \in P ; (s,x) \in \mathcal{C} {_d\times _\rho} P\} 
\]
and the action map will be then $\theta_s = \Theta(s)$. It can be shown that these data define a fibred action. 
\end{proof}

We want to define partial actions for inverse categories in Exel's \cite{ruylivro} fashion, but the structure on an inverse category may have several idempotents; so the model for our definition comes from the partial category actions introduced by Nystedt \cite{nystedt-cat-partial-acts}.

\begin{definition}\label{def-partialact-invcat}
	Let $\mathcal{C}$ be an inverse category and $(P, \leqslant)$ be a poset. Then a map  $\theta:\mathcal{C}\rightarrow \ic{P}$ is an \textit{inverse category partial action}  if it defines a pair $ (\{\theta_s\}_{s\in\mathcal{C}}, \{ D_s \}_{s\in\mathcal{C}}) $ satisfying 
	
	\begin{enumerate}[(i)]
		\item for each $s\in\mathcal{C}$ the map $\theta_s:D_{s^\circ}\subseteq P\rightarrow D_{s}\subseteq P$ is an order preserving bijection;
		
		\item $P=\displaystyle\bigcup_{{X\in\mathcal{C}^{(0)} }} D_{1_X}$;
		
		\item if $e\in \rid{C}$ the map $\theta_e$ is the identity in his domain $D_e$;
		
		\item given $s,t\in\mathcal{C}$  such that $s\leqslant t$, then $D_{s^\circ}\subseteq D_{t^\circ}$ and ${\theta_t}_{|D_{s^\circ}}=\theta_s$;
		
		\item if $s\in\mathcal{C}$, then $D_s\subseteq D_{ir(s)}$;
		
		\item if there exists the composition $st$ in $\mathcal{C}$, then $\theta_{s}(D_{s^\circ}\cap D_t) = D_s\cap D_{st}$; and $\theta_s(\theta_t(x))=\theta_{st}(x)$, for $x\in \theta_{t^\circ}(D_t\cap D_{s^\circ})$.
		
	\end{enumerate}
	
	In addition: $\theta$ is a \textit{global action} if $D_s=D_{ir(s)}$ for all $s\in \mathcal{C}$.
\end{definition}

The next proposition exhibits a relation between the Definition \ref{def-partialact-invcat} and the  Definition \ref{defi-invcat-actions-via-symm}.

\begin{proposition}\label{propo-relation-defi-symm-partialaction}
	Let $\theta: \mathcal{C}\rightarrow \iicp$ be an action by symmetries, then $\theta$ induces a family of maps $(\{\theta_s\}_{s\in\mathcal{C}}, \{ D_s \}_{s\in\mathcal{C}})$ such as in Definition \ref{def-partialact-invcat} with $D_s=D_{ir(s)}$ for each $s\in\mathcal{C}$.
\end{proposition}

\begin{proof}
	
	Given $(s:X\rightarrow Y)\in \mathcal{C}$, from the definition of functor and $\iicp$ we get $\theta_s: dom(\theta_s)\subseteq \theta(X) \rightarrow ran(\theta_s)\subseteq \theta(Y),$ where $\theta_s=\theta(s)$.

 Given that $\theta_{s^{\circ}} = (\theta_s)^{-1}$ (Proposition \ref{prop.functors.inverse.categories}) it follows that $dom(\theta_s)=ran(\theta_{s^\circ})$. In particular, if $e^2=e$ in $\mathcal{C}$, then $\theta_e=\theta_e^{-1}$ and it is the identity map of $dom(\theta_e)=ran(\theta_e)$.

	As $\mathcal{C}$ is an inverse category, $s=ss^\circ s$ for all $s\in\mathcal{C}$ and as $\theta$ is a functor, we obtain $\theta_s=\theta_{ss^{\circ}}\theta_s = \theta_s \theta_{s^{\circ}s}.$ Hence 	$\theta_s=\theta_s\theta_{s^{\circ}s}={\theta_s}\mid_{ran(\theta_{s^{\circ}s})\cap dom(\theta_s)}$, and it follows that $ran(\theta_{s^{\circ}s})\cap dom(\theta_s)=dom(\theta_s)$, \textit{i.e.} that $dom(\theta_s)\subseteq dom(\theta_{s^{\circ}s})=ran(\theta_{s^{\circ}s}).$

	Using that $dom(\theta_{s^{\circ}})=ran(\theta_s)$, we have that $
	\theta_{s^{\circ}s} = \theta_{s^{\circ}}\theta_s = {\theta_{s^\circ} \mid }_{dom(\theta_{s^{\circ}})}\circ {\theta_s \mid}_{dom(\theta_s)}$
	and since the domain of the last composition is the domain of  $\theta_s$, it follows that $dom(\theta_s)= dom(\theta_{s^{\circ }s}).$ 
 	
	Therefore we have the equalities $dom(\theta_s)= d(\theta_{s^{\circ}s}) = ran(\theta_{s^{\circ}s}),$ and $ran(\theta_s) = ran(\theta_{ss^{\circ}})=dom(\theta_{ss^{\circ}}).$
	
	Now if we define $D_s=ran(\theta_s)$, which implies that $D_{s^\circ}=dom(\theta_s)$, we have the following:
	\begin{itemize}
		\item by construction, item (i) is satisfied;
		\item as $D_{1_X}=ran(\theta_{1_X})=dom(\theta_{1_X})$, each $D_e$ is an ideal and $P=\bigcup_X dom(\theta_{1_X})$ (because $\theta$ is an action by symmetries), (ii) holds; 
		\item if $e^2=e$ the map $\theta_e$ is the identity map in its domain $D_e$, so (iii) holds;
		\item since $D_s=ran(\theta_s)= ran(\theta_{ss^\circ})$, we have $D_s=D_{ir(s)}$ for all $s$, (iv) holds;
		\item if $s,t$ are arrows such that $s\leqslant t$, then $s=ts^\circ s$; combining this equality with the fact that $D_{s^\circ}=dom(\theta_s)=D_{s^\circ s}$ we have (vi);
		\item suppose $x\in D_s\cap D_t$, then there exists $y\in D_{t^\circ}$ such that $x=\theta_t(y)$; hence $\theta_s(x)=\theta_{st}(y)\in D_s\cap D_{st}$ and therefore $\theta_s(D_{s^\circ}\cap D_t)\subseteq D_s\cap D_{st}.$ It follows then that $\theta_{s^{\circ}}(D_s\cap D_{st})\subseteq D_{s^\circ} \cap D_{s^\circ st}=D_{s^\circ}\cap D_t$ and applying $\theta_s$, we obtain $D_s\cap D_{st}\subseteq \theta_s(D_{s^\circ}\cap D_t).$
		
	\end{itemize}
	
	Therefore the proof is complete.
\end{proof}

Fixing the notations: given an inverse category $\mathcal{C}$  and a poset $P$ ,
\begin{itemize}
	\item the symbol $\overline{\theta}:(\mathcal{C},(\:\:)^\circ)\curvearrowright \iicp$, denotes a global action by symmetries, and
	\item the symbol $\theta:(\mathcal{C},(\:\:)^\circ)\curvearrowright_p \iicp$ denotes a partial action by symmetries.
\end{itemize}

About the terminology:  we deal only with actions on posets, so we will say action by symmetries and suppress the ``ordered''.

\begin{remark}
\label{remark-restriction-global-action}
As Nystedt \cite{nystedt-cat-partial-acts}, Gilbert \cite{gilbert-pdgexp} and Bagio et al  \cite{bagio2010partial} remark, a global action of a groupoid can be restricted to a partial action, and essentially the same may be done for inverse categories: we can restrict a global action of an inverse category to a partial action in the sense of Definition \ref{def-partialact-invcat}.
Indeed, let $\overline{\theta}:(\mathcal{C},(\:\:)^\circ)\curvearrowright \iicp$ a global action by symmetries on $P$, where for each $s$ we have a order preserving bijection $\overline{\theta}_s: \overline{D}_{s^\circ} \to \overline{D}_{s}$. If $Q\subseteq P$ is an ideal then there is a partial action $\theta: (\mathcal{C},(\:\:)^\circ)\curvearrowright_p \mathcal{I}_{ic}(Q)$ induced by the previous one in the following fashion: for each $s\in\mathcal{C}$ define $D_s = (Q \cap \overline{D}_s) \cap \overline{\theta}_s (Q \cap \overline{D}_{s^\circ})$, which is an ideal of $Q$, and take the order isomorphism $\theta_s = \overline{\theta}_s \mid_{D_{s^\circ}}: D_{s^\circ} \to D_s$. The assignment $s \mapsto \theta_s$ defines a partial action $\theta:  \mathcal{C} \to  \mathcal{I}_{ic}(Q)$. 
\end{remark}


\subsection{Bernoulli actions}

We will introduce Bernoulli actions of inverse categories. These actions will escort us for the rest of this article, and are our major application of the theory that we developed of inverse category (partial) actions. 

Let $(\mathcal{C},(\:\:)^{\circ})$ be an inverse category. Given $X\in\mathcal{C}^{(0)}$ and $e\in \ridx{C}{X}$ (see Definition \ref{defi-Rid-inner-outer-source-target}), we define the set $$P_{e,X}(\mathcal{C}):= \{A\subset \mathcal{C}; \mid A\mid < \infty,\:   or(a)=X, \: ir(a)=e \: , \forall a\in A  \}.$$

Then we define the following sets
$$P(\mathcal{C}):=\bigcup_{\substack{X\in\mathcal{C}^{(0)} \\ e \in \ridx{C}{X}  }}P_{e,X} \text{  and } P_\circ(\mathcal{C}):=\{A\in P(\mathcal{C}); A\cap \rid{C} \neq \emptyset  \}.$$

An equivalent characterization is possible if we use costars and right relations. For instance, if $X\in\mathcal{C}^{(0)}$ and $e \in \ridx{C}{X}$, then $A \in P(\mathcal{C})$ if and only if  $ \mid A\mid < \infty,\: A\subset Costar(X), \: A\subset \mathscr{R}_e$, and $A\in P_\circ(\mathcal{C})$ if and only if $A \in P(\mathcal{C})$ and $A\ni e$.

$P(\mathcal{C})$ and $P_\circ(\mathcal{C})$ carry a natural partial order, which is inspired by the partial order defined in \cite{lawsonmarstein-expansions}).

\begin{definition}\label{defi-partial-order-Psets-invcat}
	Let $\mathcal{C}$ be an inverse category and consider $A,B\in P(\mathcal{C})$ such that $A \in \pex{e}{X}$ and $B \in \pex{f}{Y}$, for $X,Y\in\mathcal{C}^{(0)}, e \in \ridx{C}{X}$ and $ f \in \ridx{C}{Y}$; we define an order on $P(\mathcal{C})$ by  $$A\leqslant B \iff X=Y, \: e\leqslant f, \: eB\subseteq A,$$ where the order on  $\ridx{C}{X}$  is the natural order of inverse categories which was stated in Proposition \ref{prop-inv-cat-ord}. Notation: $(P(\mathcal{C}), \leqslant)$.
\end{definition}

Clearly, this order restricts to an order in $P_\circ(\mathcal{C})$ and, moreover, $P(\mathcal{C})$ is an order ideal of $P_\circ(\mathcal{C})$.

\begin{definition}
	The (\textit{global}) \textit{Bernoulli fibred action} of an inverse category $\mathcal{C}$ on $P(\mathcal{C})$ is the pair $(\varepsilon, \mathfrak{B})$, where
	
	\begin{itemize}
		\item $\varepsilon=(o\varepsilon,i\varepsilon)$ is the moment map $\varepsilon: P(\mathcal{C})\rightarrow \mathcal{C}^{(0)}\times \rid{C}$, which is defined as follows : if $A \in \pex{e}{X}$ then $\varepsilon(A)=(o\varepsilon(A),i\varepsilon(A))=(X, e)$;
		\item the action map $\mathfrak{B}:\mathcal{C} {_d\times _\varepsilon}P(\mathcal{C})\rightarrow P(\mathcal{C}) $ is given by  $(s,A)\mapsto \mathfrak{B}(s,A)=\mathfrak{B}_s(A):= sA$.
	\end{itemize}
\end{definition}

Let us check that this is indeed a fibred action. 

Let $A\subset Costar(X)$, $A\subset \mathscr{R}_e$, $e \in \ridx{C}{X}$.
The action map is well defined: if there exists $\mathfrak{B}_s(A)=sA$, then $sA$ is a finite set such that $sA\subset Costar(or(s))$ and $sA\subset \mathscr{R}_{ses^{\circ}}$, that is, $sA \in P_{or(s), ses^{\circ}} (\mathcal{C}) \subset P (\mathcal{C})$.

By the definition of the order in $P(\mathcal{C})$ and of the moment map, the inner moment map preserves order.

The action map also preserves order. 
Let $A \in P_{e,X} (\mathcal{C})$ and $B \in P_{f,Y} (\mathcal{C})$ such that $A \leq B$; let $s,t\in \mathcal{C}$, $s:X\rightarrow U$ and $t: Y\rightarrow V$, such that $s\leqslant t$ and that $\mathfrak{B}_s(A)$ and $\mathfrak{B}_t(B)$ are defined; we will show that $\mathfrak{B}_s(A)\leqslant\mathfrak{B}_t(B)$. In fact, 
 \begin{itemize}
		\item since $o\varepsilon(sA)= or(s)=U$, $o\varepsilon(tB)=or(t)=V$, and from $s\leqslant t $ we have that the composition $s^\circ t $ exists, so $or(t)=od(s^\circ)$ which implies that $U=V$;
		\item as $i\varepsilon(sA)=ses^\circ$ and $i\varepsilon(tb)=tft^\circ$, and from $s\leqslant t $ and $e\leqslant f$ we have $ses^\circ\leqslant sfs^\circ \leqslant tft^\circ$;
		\item as $eB\subseteq A$,  $e\leqslant id(s)$ and $f\leqslant id(t)$, $e\leqslant f$ and $s\leqslant t$   we can show that $sef=ses^\circ sf=ses^\circ st^\circ tf=sefs^\circ st^\circ t=ses^\circ t$; hence $ses^\circ(tB)=seB\subseteq sA$.
	\end{itemize}
	The conclusion is that $sA\leqslant tB$, \textit{i.e.} $\mathfrak{B}_s(A)\leqslant \mathfrak{B}_t(B)$.

Axioms (I)-(III) from Definition \ref{defi-invcat-action} can be verified in an analogous manner. 

Now that we have constructed the global Bernoulli action, we may restrict it to the poset $P_\circ(\mathcal{C})$ and thus acquire a partial action. To achieve this goal, we turn to  the global action by symmetries associated to $(\varepsilon, \mathfrak{B})$.

The \textit{Bernoulli action by symmetries} is the functor $\mathfrak{B} : \mathcal{C} \to \ic{P(\mathcal{C})}$ that takes $s\in\mathcal{C}$ to the map $$\mathfrak{B}_s: \overline{ {D }}_{s^\circ}\rightarrow \overline{ {D }}_s \text{  with  }  A\mapsto \mathfrak{B}_s(A)=sA,$$ 
where 
\[
\overline{ {D} }_{s} := \{B \in P(\mathcal{C}); o\varepsilon(B)=or(s),\: i\varepsilon(B)\leqslant ir(s) \}.\]
Using the definitions of inverse category and of the moment map $\varepsilon$, the domain $\overline{ {D }}_{s^\circ}$ may be described also in terms of $s$ as $$\overline{ {D} }_{s^\circ} := \{A\in P(\mathcal{C}); o\varepsilon(A)=od(s),\: i\varepsilon(A)\leqslant id(s) \}.$$

Applying the method of Remark \ref{remark-restriction-global-action}, \textit{i.e.} the restriction of a global action to a partial action, we can define the \textit{Bernoulli partial action by symmetries} $\mathfrak{b}_s:=(\mathfrak{B}_s)_{|{D_s}}\mathfrak{b}_s:D_{s^\circ}\rightarrow D_s$, where
\begin{align*}
D_s& =\{B \in P(\mathcal{C}); o\varepsilon(B)=or(s), \: i\varepsilon(B)\leqslant ir(s), \: B\ni i\varepsilon(B)s ,\: i\varepsilon(B) \}, \\ 
D_{s^\circ} & =\{A\in P(\mathcal{C}); o\varepsilon(A)=od(s),\: i\varepsilon (A)\leqslant id(s), \: A\ni i\varepsilon(A)s^\circ, i\varepsilon(A)\}. 
\end{align*}

We called $\mathfrak{B}$ a global action and $\mathfrak{b}$ a partial action, and in the next proposition we show that these actions do define a global action and a partial action, respectively.

\begin{lemma}\label{lemma-global-partial-Bernoulli-proof}
	Given an inverse category $(\mathcal{C}, (\:\:)^\circ )$ and the posets $P(\mathcal{C})$ and $P_\circ(\mathcal{C})$, the pairs  $(\{{\mathfrak{B}}_s\}_{s\in\mathcal{C}}, \{{\overline{ {D }}_s}\}_{s\in\mathcal{C}})$  and $(\{{\mathfrak{b}}\}_{s\in\mathcal{C}}, \{{{ {D }}_s}\}_{s\in\mathcal{C}})$ are, respectively, a global and a partial inverse category actions. In addition, $\mathcal{C}\cdot P_\circ(\mathcal{C})= P(\mathcal{C})$.
\end{lemma}

\begin{proof}
	We will  check the axioms of Definition \ref{def-partialact-invcat} for  $\mathfrak{B}$. The other case is analogous.
	
	\begin{enumerate}[(i)]
		\item From Lemma \ref{lemma-invcat-action-symmetries}, each map $\mathfrak{B}_s : \overline{D}_{s^\circ} \to \overline{D}_s$ is an order preserving bijection.
		
		\item By construction $P(\mathcal{C})=\displaystyle\bigcup_{X\in\mathcal{C}^{(0)}}\overline{ {D }}_{1_X}$;
		\item If $A\in\overline{ {D }}_e$, by definition $i\varepsilon(A)\leqslant e$, so $eA=A$.
		\item Suppose $s\leqslant t $, by inverse category properties we have $s=ts^\circ s$ and $id(s)\leqslant ir(t)$. The existence of $ts^\circ $ implies that $od(t)=or(s^\circ)=od(s)$ and it follows that  $\overline{ {D }}_{s^\circ}\subset \overline{ {D }}_{t^\circ}$. Moreover, $i\varepsilon(A)\leqslant id(s)$ and $s=ts^\circ s$, so $i\varepsilon(A)\leqslant s^\circ st^\circ$, and for $A\in \overline{ {D }}_{s^\circ}$, thus $tA=ti\varepsilon(A)A=t(s^\circ s t^\circ t i\varepsilon(A))A=ts^\circ sA =sA.$
		\item  Since $or(ss^\circ)=or(s)$ and $ir(ss^\circ)=ir(s)$, we can see that $\overline{ {D }}_s = \overline{ {D }}_{ss^\circ}$.
		
		\item Assume that $st$ exists in $\mathcal{C}$ and take $A\in\overline{ {D }}_{s^\circ}\cap \overline{ {D }}_t$. Last assumption asserts that $o\varepsilon(A)=od(s) \text{  and  } o\varepsilon(A)=or(t),$ and $i\varepsilon(A)\leqslant is(A) \text{  and  } i\varepsilon(A)\leqslant ir(t).$  If we calculate $sA$, observe that $o\varepsilon(sA)= or(s)=or(st),$ $i\varepsilon(sA)=si\varepsilon(A)s^\circ \leqslant s(id(s))s^\circ=ir(s)$ and
		$i\varepsilon(sA)\leqslant s(ir(t))s^\circ = ir(st).$ Therefore $sA\in \overline{ {D }}_s\cap \overline{ {D }}_{st}$. The last claim follows from the fact that $\mathcal{B}$ is the action map of a fibred action.	\end{enumerate} 
	
	By proving the items above we have just shown that $\mathfrak{B}$ defines a global inverse category action by symmetries. \end{proof}

Notation: $\mathfrak{B}:(\mathcal{C}, (\:\:)^\circ )\curvearrowright  \mathcal{I}_{ic}(P(\mathcal{C}))$ is a global action, and $\mathfrak{b}:(\mathcal{C}, (\:\:)^\circ )\curvearrowright_p \mathcal{I}_{ic}(P_\circ(\mathcal{C}))$ is a partial action.

\

There are two more actions we would like to define. These new actions will arise if we change the inequality to an equality in both global and partial Bernoulli actions domains. We were inspired by O'Carroll's strict inverse semigroups as in \cite{ocarroll-strongI}, and its relation to partial actions proposed by Khrypchenko \cite{mykola-partial}.

To avoid cluttering the text with too much information we introduce a new notation. Indeed, let $\mathcal{C}$ be an inverse category and let $P$ be a poset, we define the set $$\mathcal{C} {_d\overline{\times} _\rho}P:=\{(s,x)\in \mathcal{C}\times P; o\rho(x)=od(s), \: i\rho(x)=id(s)\}.$$

Next we define the \textit{strict Bernoulli actions} in their fibred versions. Let $\mathcal{C}$ be an inverse category and consider the sets $P(\mathcal{C})$ and $P_\circ(\mathcal{C})$.

\begin{definition}\label{def-strict-ic-actoin}
	The \textit{strict global fibred Bernoulli action} of $\mathcal{C}$ on $P(\mathcal{C})$ is defined by the maps
	\begin{itemize}
		\item  $\varepsilon: P(\mathcal{C})\rightarrow \mathcal{C}^{(0)}\times $, with  $P(\mathcal{C})\mapsto \varepsilon(A)=(o\varepsilon(A),i\varepsilon(A))=(X, e)$;
		\item $\mathfrak{sB}:\mathcal{C} {_d\overline{\times} _\varepsilon}P(\mathcal{C})\rightarrow P(\mathcal{C}) $ with $(s,A)\mapsto \mathfrak{sB}(s,A)=\mathfrak{sB}_s(A):= sA$.
	\end{itemize}
\end{definition}

As we did with the global fibred action $\mathfrak{B}$, we can define the strict global action by symmetries and then restrict  it (following the lines of Remark \ref{remark-restriction-global-action}) to a partial strict action. 

Given $s\in\mathcal{C}$, the \textit{strict global Bernoulli action by symmetries} is the map ${\mathfrak{sB}_s:\overline{D}^m_{s^\circ}\rightarrow \overline{D}^m_s}$ where  
$$\overline{D}^m_s:=\{A\in P(\mathcal{C}); o\varepsilon(A)=or(s),\: i\varepsilon(A)= ir(s)\}.$$

Restricting the moment map $\varepsilon$ and the action map $\mathfrak{sB}$ to the set $P_\circ(\mathcal{C})$, we can define for each $s\in \mathcal{C}$ the \textit{strict partial fibred Bernoulli action} $(\epsilon, \mathfrak{sb})$ of $\mathcal{C}$ on $P_\circ({\mathcal{C}})$. Moreover, we have the \textit{strict partial Bernoulli action by symmetries} $\mathfrak{sb}_s: \overline{D}^m_{s^\circ}\rightarrow \overline{D}^m_s$, where $$\overline{D}^m_s=:\{A\in P(\mathcal{C}); o\varepsilon(A)=or(s),\: i\varepsilon(A)= ir(s), \: A\ni ss^\circ,s \}.$$

\section{Semidirect product}\label{section-sd-product}

We will associate an inverse category to each inverse category fibred
action via an adaptation of the categorical semidirect product from Tilson-Steinberg \cite{steinberg-categories-as-alg-II}.  Next, we specialize it to our study of Bernoulli’s actions.

\subsection{The semidirect product of a fibred action}

Let $(\rho,\theta): (\mathcal{C}, (\: \:)^{\circ})\curvearrowright (P,\leqslant)$ be a fibred action of an inverse category on a poset. We will associate an inverse category to this datum, the semidirect product of $P$ by $\mathcal{C}$, which we will define in terms of the set of its arrows. 

\begin{definition}\label{defi-invcat-semidirect-prodc}
	The \textit{semidirect product} determined by $(\rho,\theta)$ is the  category $P\rtimes_{(\rho,\theta)}\mathcal{C}$ with the following structure: 
	\begin{itemize}
		\item  \underline{arrows}: $P\rtimes_{(\rho,\theta)}\mathcal{C}:= \{(x,s)\in P\times \mathcal{C}; o\rho(x)=or(s), \: i\rho(x)\leqslant ir(s) \};$
		\item  \underline{objects}: $( P\rtimes_{(\rho,\theta)}\mathcal{C})^{(0)}:=\{ (x,1)\in P\times \rid{C}; o\rho(x)=or(1_{o\rho(x)}), \: i\rho(x)\leqslant 1_{o\rho(x)} \}$;
		\item  \underline{composition}: $(x,s)(y,t)=(x,st)$ if, and only if, $x=\theta_s(y)$ and $st$ is defined in $\mathcal{C}$.
	\end{itemize} 
\end{definition}
Note that
\[
(x,1_{o\rho(x)})(x,s) = (x,s) = (x,1_{o\rho(x)})(\theta_{s^{\circ}} (x),1_{o \rho(\theta_{s^{\circ}} (x))}), 
\]
that is, 
 $(x,s)$ is an arrow from $(\theta_{s^{\circ}} (x),1_{o \rho(\theta_{s^{\circ}} (x))})$ to $(x,1_{o\rho(x)})$. A routine verification shows that $P\rtimes_{(\rho,\theta)}\mathcal{C}$ is indeed a category and, by its definition, its objects are in a bijective correspondence with the set $P$ via  $(x,1_{o \rho(x)}) \in P\rtimes_{(\rho,\theta)}\mathcal{C} \mapsto x \in P$.

It is important to describe the set of idempotents of this category:
$$\rid{ P\rtimes_{(\rho,\theta)}\mathcal{C}}:=\{ (x,e)\in P\times \rid{C}; o\rho(x)=or(e), \: i\rho(x)\leqslant e \}.$$

When there is no possibility of misunderstanding, we will write just $P\rtimes \mathcal{C}$.

\begin{lemma}\label{lemma-invcat-sdp-is-invcat}
Given a fibred action $(\rho,\theta): (\mathcal{C}, (\: \:)^{\circ})\curvearrowright (P,\leqslant)$, the category $P\rtimes \mathcal{C}$  is an inverse category with $(x,s)^{\circ}:=(\theta_{s^{\circ}}(x),s^\circ)$.
\end{lemma} 

\begin{proof}
	We will verify the axioms of (ii) in Proposition \ref{prop-defi-equiv-invcat}. The verification that $(x,s) \mapsto (\theta_{s^{\circ}} (x), s^{\circ})$ defines a contravariant functor, that it fixes objects and that the first two equalities are direct computations, so we will dedicate attention only to proving that idempotent arrows commute.  Indeed: let $(x,s)$ and $(y,t)$ be elements of $P\rtimes \mathcal{C}$.
By construction, the composition $[(x,s)(x,s)^{\circ}][(y,t)(y,t)^{\circ}]=(x,ss^\circ)(y,tt^\circ)$ is defined if and only if   $x=\theta_{ss^{\circ}}(y)$, and it follows that $[(y,t)(y,t)^{\circ}][(x,s)(x,s)^{\circ}]$ is defined if and only if $ y=\theta_{tt^{\circ}}(x).$ 		
Now $\mathcal{C}$ is an inverse category, hence  $ss^\circ tt^\circ  =tt^\circ ss^\circ$; from $  x=\theta_{ss^{\circ}}(x)$  and $y=\theta_{tt^{\circ}}(y)$ we deduce $\theta_{tt^{\circ}}(x)=\theta_{tt^{\circ}}(\theta_{ss^{\circ}}(x))= \theta_{ss^{\circ}}(\theta_{tt^{\circ}}(x)) =\theta_{ss^{\circ}}(y). $
		Therefore  
  $$
  \begin{array}{rcl}
  (x,ss^\circ)(y,t^\circ)& =& (x, ss^\circ tt^\circ) = (\theta_{ss^{\circ}}(y), tt^\circ ss^\circ)=\\
  &=&(\theta_{tt^{\circ}}(x),tt^\circ ss^\circ   )= (y, tt^\circ ss^\circ) = (y,tt^\circ)(x,ss^\circ).
  \end{array}$$ 
  	
Hence $P\rtimes \mathcal{C}$ is also an inverse category. \end{proof}
We remark that, following  Definition \ref{defi-Rid-inner-outer-source-target}, the inner source of $(x,s)\in P\rtimes \mathcal{C}$ is the idempotent morphism $id(x,s)=(\theta_{s^{\circ}}(x), s^\circ s)$, and
the inner target of $(x,s)$ is the idempotent morphism $ir(x,s)=(x,ss^\circ)$.


\subsection{Szendrei expansions}

The machinery from the last paragraphs is put in motion in this subsection, where we will define the semidirect product determined by the Bernoulli actions, which we call the Szendrei expansions of an inverse category.

\begin{definition}\label{defi-exp-invcat}
	Let $(\varepsilon,\mathfrak{B}):(\mathcal{C}, (\:\:)^\circ)\curvearrowright (P(\mathcal{C}),\leqslant)$ be the Bernoulli fibred action. The \textit{Szendrei expansion} of the inverse category $\mathcal{C}$ is the inverse category $\overline{Sz}(\mathcal{C}):=P(\mathcal{C})\rtimes_{(\varepsilon,\mathfrak{B})}\mathcal{C}.$
\end{definition}

In \cite{birget1984almost}, Birget and Rhodes introduce an \textit{expansion} of a semigroup as a functor $F$ from a category of semigroups to itself (or to a particular subcategory) equipped with a natural transformation $\eta$ from $F$ to the identity functor such that  $\eta_S : F(S) \to S$ is surjective for every semigroup $S$. 
Let us introduce the obvious general concept of an expansion.
\begin{definition}
    An \emph{expansion} of a category $\calC$ is a functor $F : \calC \to \calC$ such that there exists a natural transformation $\eta: F \to Id_{\calC}$ such that  $\eta_x : F(x) \to x$ is surjective for every object $x \in \calC^{(0)}$.
\end{definition}
 Szendrei's expansion $\overline{Sz}(\calC)$ is  an expansion in this sense.
Let $\invcat$ denote the category whose objects are  inverse categories and whose morphisms are functors. 
\begin{proposition}
    Szendrei's expansion of inverse categories defines an expansion $\overline{Sz}: \invcat \to \invcat$. 
\end{proposition}
\begin{proof}

    Let $\calC, \calD$ be two inverse categories and let $f: \calC \to \calD$ be a functor. Given  a subset $A \subset \calC^{(0)}$, let $f(A)$ be the set $f(A) = \{f(s); s \in A\} \subset \calD^{(0)}$. Since functors preserve the inverse structure, given a morphism $(A,s) \in \overline{Sz}(\calC)$, if $(A,s) \in P_{e,X} (\calC)$ then $(f(A),f(s)) \in P_{f(e),f(X)} (\calD)$; a straightforward computation shows that  $$\overline{Sz}(f) : \overline{Sz}(\calC) \to \overline{Sz}(\calD), \ \ \ \overline{Sz}(f) (A,s) = (f(A),f(s))$$    is a functor. Moreover,  if $g : \calD \to \mathcal{E}$ is  another functor between inverse categories then $\overline{Sz}(g \circ f) = \overline{Sz}(g )\circ\overline{Sz}( f).$ Finally, the natural transformation from $\overline{Sz}$ to $Id_{\invcat}$ is given by  $$\eta_{\mathcal{C}} :\overline{Sz}(\mathcal{C}) \to \mathcal{C},  \ \ \ \ \eta_{\mathcal{C}} (A,s) = s.$$

\end{proof}
\begin{remark}
$\overline{Sz}(\calC)$ comes equipped with the natural order of an inverse category: if $A \in \pex{e}{X}$, $B \in \pex{f}{Y}$ and $(A,s), (B,t) \in \overline{Sz}(\calC)$, then $(A,s) \preceq (B,t)$ if and only if $(A,s) = (A,s)(s^\circ A, s^\circ ) (B,t)$. Simple computations show then that 
\begin{align*}
(A,s) \preceq (B,t) & \iff  X=Y, \: e\leqslant f, \: eB = A, \: s \leq t \\
& \implies A \leq B \textrm{ and } s \leq t, 
    \end{align*}
which shows that if $(A,s) \preceq (B,t)$ then we also have that  $(A,s) \leq (B,t)$ with respect to the ordering induced by the \textit{product order} on $P(\calC) \times \calC$, which is precisely 
\[
(A,s) \leq (B,t) \iff A \leq_{P(\calC)} B \textrm{ and } s \leq_{\calC} t,   
\]
so that the product order refines the natural order on $\overline{Sz}(\calC)$.     
\end{remark}

In what follows we will show that the order induced in $\overline{Sz}(\calC)$ by the product order is compatible with the category structure in a precise manner. 
The next definition is inspired by the definition of ordered groupoids by Lawson \cite{lawsonlivro} and Hollings' \cite{hollings-thesis-partial} definition of ordered categories.

\begin{definition}\label{defi-ordered-invcat}
	We say that an inverse category $(\mathcal{C}, (\:\:)^\circ)$ endowed with a partial ordering $\leqslant$ is an \textit{ordered inverse category} when for $s,t,s',t' \in \mathcal{C}$ and idempotent arrows $e,f\in\mathcal{C}$,
	\begin{enumerate}[(I)]
		\item if  $s\leqslant s'$ and $t\leqslant t'$, then $ st\leqslant s't'$;
		\item if $s\leqslant t$, then $id(s)\leqslant id(t)$ and  $ir(s)\leqslant ir(t)$;
		\item if $e\leqslant id(s)$, then there exists a unique morphism $_{e|}s \in \mathcal{C}$  s.t. $ _{e|}s\leqslant s$  and $id(_{e|}s)=e$;
		\item if $f\leqslant ir(t)$, then there exists a unique  morphism $t_{|f} \in \mathcal{C}$  s.t. $t_{|f}\leqslant t$ and $ir(t_{|f})=f$.
	\end{enumerate}
\end{definition} 

The arrow of item (III) is called the \textit{restriction} to $e$ and the arrow of (IV) is the \textit{corestriction} to $f$.

\begin{remark}\label{remark-invcat-nat-orderedcat}
	Every inverse category is an ordered inverse category by its natural ordering where the restriction and the corestriction are defined as follows: for the latter, if $e\leqslant id(s)$ and $f\leqslant ir(t)$ define $_{e|}s:=se$ and $t_{|f}:= ft$.
\end{remark}

\begin{definition}\label{defi-ord-for-szendrei-exp}
The partial order $\leqslant$ on  $\overline{Sz}(\mathcal{C})$ is the one induced by the product order on $P(\calC) \times \calC$, given by 
 $(A,s)\leqslant (B,t)$ if and only if $A\leqslant_{P(\mathcal{C})}B$ and $ s\leqslant_{\mathcal{C}}t.$
\end{definition}

Let us open up this definition: if $(A,s),(B,t)\in \overline{Sz}(\mathcal{C})$ are pairs with
$A\subset Costar(X),\: A\subset \mathscr{R}_e \text{ and } B\subset Costar(Y),\: B\subset \mathscr{R}_f,$ where $X,Y\in \mathcal{C}^{(0)}$ and $e,f \in \rid{\mathcal{C}}$ then, by Definition \ref{defi-partial-order-Psets-invcat} and Proposition \ref{prop-inv-cat-ord}, $$(A,s)\leqslant (B,t) \iff X=Y,\:  e\leqslant f,\: eB\subseteq A, \text{ and } \: s=ss^\circ t.$$

\begin{lemma}\label{lemma-szendrei-exp-is-orderedcat}
	The Szendrei expansion $(\overline{Sz}(\mathcal{C}),(\:\:)^\circ)$ with $\leqslant$ is an ordered inverse category.
\end{lemma}

\begin{proof}
	We must verify each condition of Definition \ref{defi-ord-for-szendrei-exp}. As the first two conditions are direct computations, we will devote our attention to the third item and indicate the main idea of the last one.

 We will deal with restrictions and corestrictions, in this order.

Suppose  $(E,f)$ is an idempotent in $\overline{Sz}(\mathcal{C})$ and $(A,s)\in\overline{Sz}(\mathcal{C})$; then there exist objects $X,Y$ in $\mathcal{C}$ and  such that $E\subset Costar(Y)$, $ E\subset \mathscr{R}_i$, $f: Y \rightarrow Y$, $i\leqslant f $    and   $ A\subset Costar(X)$, $A\subset \mathscr{R}_e$,  $s:U\rightarrow X$, $e\leqslant ss^\circ.$

Suppose that $(E,f)\leqslant id(A,s)$, which means that $(E,f) \leqslant (s^\circ A, s^\circ s)$, that is,  $E\leqslant s^\circ A$  and $f \leqslant s^\circ s$. We will prove that $ _{(E,f)|}(A,s):=(sE,sf)$   is the restriction of   $(A,s) $ to $(E,f)$.

We begin by proving that $(sE,sf)$ is a well-defined element of the Szendrei expansion.
Since $s^\circ A\subset 
Costar(U)$ and $E\leqslant s^\circ A$, we conclude that $Y=U$ and therefore the composition $sf$ exists in $\mathcal{C}$. 
The set $sfE$ is well defined since $od(sf)=od(f)=o\varepsilon(E)=Y.$ 
   The pair $(sE, sf)$ is an element of $\overline{Sz}(\mathcal{C})$, due to the fact that $or(sf)=or(s)=X=o\varepsilon(sE)$, and $i\leqslant f$ implies  $i\varepsilon(sE) = sis^\circ \leqslant sfs^\circ = ir(sf).$

In what follows we will need the inner source of $(sE,sf)$, which is given by $id(sE,sf)= ((sf)^\circ sE, (sf)^\circ (sf))= (E,f)$. In fact, from $f\leqslant s^\circ s$, $E\subset \mathscr{R}_i$ and $i\leqslant f$ we get  $(sf)^\circ  sE = fs^\circ sE=fE=fiE=iE=E$; furthermore, $(sf)^\circ (sf)=fs^\circ s f=f$.

In order to prove that $(sE,sf)\leqslant (A,s)$, we need to show that $sE\leqslant A$ and $sf \leqslant s $, that is, to show that $ o\varepsilon(sE) = o\varepsilon(A)$, $i\varepsilon(sE)\leqslant i\varepsilon(A)$, $i\varepsilon(sE)A\subset sE$ and that $sf\leqslant s$. Indeed, respectively: the outer moment map condition follows from $o\varepsilon(sE)=or(s)=o\varepsilon(A)$;
the initial hypothesis $E\leqslant s^\circ A$ implies $i\leqslant s^\circ e s$, so $i = s^\circ esi$, and $ss^\circ es=es$ implies $sis^\circ=esis^\circ$; therefore $i\varepsilon(sE)=sis^\circ \leqslant e$;
since $E\leqslant s^\circ A$, we have that $is^\circ A\subset E$, which yields  $sis^\circ A\subset sE$; finally, the inequality $sf \leq s$ is a consequence of $(sf)(sf)^\circ s=sfs^\circ s= sf$, since $f\leqslant s^\circ s$.

It remains to be shown that the restriction is unique. Let $(B,t)$ be an arrow in $\overline{Sz}(\mathcal{C})$ such that $(B,t)\leqslant (A,s) \text{  and  }  (E,f) = id(B,t)$, that is,  $B\leqslant A, \: t\leqslant s, \: t^\circ B= E \text{  and  } t^\circ t =f.$ Using the last item of Proposition \ref{prop-inv-cat-ord}, from  $t\leqslant s$ we obtain $t=st^\circ t =sf.$ Furthermore, as $(B,t)\in \overline{Sz}(\mathcal{C})$, $tt^\circ B=B$, so $t^\circ B=E$ implies that $B=tt^\circ B = tE.$ Finally, $tE = sfE=sE$, because $(E,f)\in\overline{Sz}(\mathcal{C})$. In conclusion, $(B,t)=(sE, sf).$

We conclude from the previous computations that $_{(E,f)|}(A,s):=(sE,sf)$ is the restriction of the arrow $(A,s)$ to $(E,f)$.

Following similar arguments, we can verify that given an arrow $(A,s)$ and a restriction idempotent $(E,f)$ such that $(E,f)\leqslant (A,s)$, then $(A,s)_{|(E,f)}=(E,fs)$ defines the corestriction. \end{proof}

Our next move is towards defining another operation among arrows. After Dewolf-Pronk \cite{dewolf-ehresmann} Proposition 3.2, we state the next lemma.

\begin{lemma}[\cite{dewolf-ehresmann}]\label{lemma-meets-idemp-invcat}
	Let $X$ be an object of the inverse category $(\mathcal{C}, (\:\:)^\circ)$. The set $\ridx{C}{X}$ of idempotent morphisms in $X$ is a meet semilattice with respect to the natural partial order from $\mathcal{C}$ and $e\wedge f:=ef$.
\end{lemma}

It is worth mentioning that each $\ridx{C}{X}$ has $1_X$ as top element.

Continuing, we extend the pseudo product from ordered groupoids \cite{lawsonlivro} to ordered inverse categories.

\begin{definition}\label{defi-pseudoprod-invcat}
	Let $\mathcal{C}$ be an ordered inverse category such that for each object $X$ the set $\ridx{C}{X}$ is a meet semilattice. Let $s,t\in\mathcal{C}$ be arrows such that there exists $id(s)\wedge ir(t)$. The \textit{pseudo product} of $s$ and $t$ is $$s\star t:= (_{id(s)\wedge ir(t)|}s )(t_{|id(s)\wedge ir(t)}).$$
\end{definition}

\begin{remark}
	Consider an inverse category $\mathcal{C}$ with the natural partial order and let $s,t$ be arrows in $\mathcal{C}$.  
 The composition $st$ is defined if and only if the composition 
$ s^\circ stt^\circ = id(s)\wedge ir(t)$ is defined and, in this case, it follows from  Lemma \ref{lemma-meets-idemp-invcat} and \ref{remark-invcat-nat-orderedcat} that 
$s\star t = s(id(s) \wedge ir(t)) t = 
s(s^\circ stt^\circ)(s^\circ stt^\circ)t=st.$
Therefore $st$ is defined if and only if $s\star t$ is, and these elements are equal. 
\end{remark}

We turn now to study the pseudo product in $\overline{Sz}(\mathcal{C})$ where, in contrast to the previous remark,  it is a proper extension of the composition of this category. We begin by computing the wedges of idempotents; the next lemma is inspired by results of Gilbert
 \cite{gilbert-pdgexp}.

\begin{lemma}\label{lemma-Szendrei-wedge}
	Let $(E,i),(F,j)$  be idempotent morphisms in $\overline{Sz}(\mathcal{C})$, such that there exists  the composition $ij$ in $\mathcal{C}$. Then the wedge product $$(E,i)\wedge (F,j):=(i\varepsilon(F)E\cup i\varepsilon(E)F, ij)$$ is an idempotent arrow in Szendrei's expansion, and is the greatest lower bound of $(E,i)$ and $(F,j)$.
\end{lemma}

\begin{proof}
	Given $(E,i),(F,j)\in \overline{Sz}(\mathcal{C})$, suppose that $E\subset Costar(X),\: E\subset \mathscr{R}_e,\: i:X\rightarrow X,\: e\leqslant i$,   and $ F\subset Costar(Y),\: F\subset \mathscr{R}_f,\: j:Y\rightarrow Y,\: f\leqslant j .$ By hypothesis $\exists ij\implies X=Y. $ In addition, as $i\varepsilon(E)=e$ and $i\varepsilon(F)=f$, the wedge product is $(E,i)\wedge (F,j):=(fE\cup eF, ij).$ 
	
It can be shown that the wedge product is an element of $\overline{Sz}(\mathcal{C})$, and that it is indeed the greatest lower bound of $(E,i)$ and $(F,j)$.
\end{proof}

Thanks to the formula for the wedge of idempotents in a Szendrei's expansion we can now describe the pseudo product.

\begin{proposition}\label{propo-pseudoproduc-invcat}
Let $(A,s),(B,t)$ be arrows in Szendrei's expansion $\overline{Sz}(\mathcal{C})$. The pseudo product $(A,s)\star(B,t)$ is defined in  $\overline{Sz}(\mathcal{C})$ if and only if $st$ is defined in $\mathcal{C}$ and, in this case, $$(A,s)\star(B,t)= (s(i\varepsilon(B))s^\circ A\cup (i\varepsilon(A))sB, st).$$ 
\end{proposition}

\begin{proof}
	Let $(A,s), (B,t)\in\overline{Sz}(\mathcal{C})$ with $ A\subset Costar(X),\:  A\subset \mathscr{R}_e,\: s:U\rightarrow X,\: e\leqslant ss^\circ $,    and $  B\subset Costar(Y),\: B\subset \mathscr{R}_f,\: t:V\rightarrow Y,\: f\leqslant tt^\circ .$ 	

 Let  $(L,l)$ be the idempotent $(L,l):=  id(A,s)\wedge ir(B,t)$. Note that $(L,l) =  (s^\circ A, s^\circ s)\wedge (B,tt^\circ)$ and from Lemma  \ref{lemma-Szendrei-wedge} it follows  that this wedge is defined if and only if $s^\circ s tt^\circ$ is defined, which is equivalent to $st$ being defined. Moreover, we have $s^\circ A\subset \mathscr{R}_{s^\circ i\varepsilon(A) s}$, and therefore the formula of Lemma \ref{lemma-Szendrei-wedge}  yields $$(L,l)  = (\: s i\varepsilon(B)(s^\circ A)\cup s^\circ i\varepsilon(A)sB, s^\circ s tt^\circ) = 
(\: sfs^\circ A\cup s^\circ esB, s^\circ s tt^\circ) .$$

 Computing the restriction and corestriction:  as $s(s^\circ es)=es$ and $s(s^\circ s tt^\circ)= stt^\circ$ we have that   $_{(L,l)|}(A,s) = (sL, sl)= (sfs^\circ A\cup esB, stt^\circ);$ and since $lt= (s^\circ stt^\circ)t= s^\circ s t$, clearly $(B,t)_{|(L,l)}=(L,lt)=(fs^\circ A \cup s^\circ esB, s^\circ st).$ 
		
Finally, as there exists the composition of previous arrows, and $slt=s(s^\circ s tt^\circ )t=st$, we conclude that $$(A,s)\star (B,t) = (sL,sl)(L,lt)=(sL, slt) = (sfs^\circ A\cup es B, st).$$
	Last but not least, we need to ensure that  $(A,s)\star (B,t)$ is an element of $\overline{Sz}(\mathcal{C})$.
Actually, $o\varepsilon(sfs^\circ A) = or(s)=X$, $ o\varepsilon(es)=or(e)=X$ and $or(st)=or(s)=X $ implies $o\varepsilon(sfs^\circ A\cup esB)  =or(st).$ In addition, since $sfs^\circ A, esB\subset \mathscr{R}_{esfs^\circ}$, $(st)(st)^\circ =stt^\circ s^\circ$ and $f\leqslant tt^\circ$, we have that $(esfs^\circ) (st)(st)^\circ = (esfs^\circ) (stt^\circ s^\circ)= esfs^\circ $ which implies that $ i\varepsilon(sfs^\circ A\cup esB) \leqslant ir(st).$ This concludes the proof.

\end{proof}

A natural question is whether  the pseudo product extends the composition in $\overline{Sz}(\mathcal{C})$.
\begin{lemma}\label{lemma-pseudoprod-equal-compositio}
	If there exists the composition of  $(A,s)$ and $(B,t)$ in $\overline{Sz}(\mathcal{C})$, then $(A,s)\star(B,t) = (A,s)(B,t).$
\end{lemma}

\begin{proof}
	Suppose $A\subset \mathscr{R}_e$ and $B\subset \mathscr{R}_f$. By assumption $\exists (A,s)(B,t)=(A,st) \iff A=sB \text{  and  } \exists st.$ From the equality $A=sB$, and  since $sB\subset \mathscr{R}_{sfs^\circ}$, we conclude that $e=sfs^\circ$. Using this information in the formula of the pseudo product we arrive at  $(A,s)\star(B,t) = (sfs^\circ A \cup esB,st)= (eA,st)=(A,st),$ where the last passage is due to the fact that $A\subset \mathscr{R}_e$. This concludes the proof.  \end{proof}

\begin{proposition} \label{prop.szendrei.pseudo.prod}
     $(\overline{Sz}(\mathcal{C}), \star)$ is an inverse category. 
\end{proposition} 
\begin{proof}
It can be shown that the product $((A,s) \star (B,t) ) \star (C,u)$ is defined if and only if 
$ (A,s) \star ((B,t) \star (C,u))$ is defined and, 
    using the formula obtained in Proposition \ref{propo-pseudoproduc-invcat} and the same notations from its proof, both products are equal to the element $$
    (s ( f (tgt^\circ))s^\circ )A \cup es ( (tgt^\circ) B \cup ft C), stu).
    $$

    The idempotents for the pseudo product are the idempotents for the composition in $\overline{Sz}(\mathcal{C})$. It follows by Lemma \ref{lemma-pseudoprod-equal-compositio} that the arrows $(A,1_X)$ are the identities of this category.
    
For $(A,s) \in \overline{Sz}(\mathcal{C})$ we have that $(A,s)^\circ = (s^\circ A, s^\circ)$.
It can be shown that the association $(A,s) \mapsto (s^\circ A, s^\circ)$ satisfies the equality $((A,s)\star (B,t))^\circ = (B,t)^\circ \star (A,s)^\circ $ whenever the pseudo product $(A,s)\star (B,t)$ is defined; moreover, idempotents commute with respect to the pseudo product: if $(E,i)$ and $(F,j)$ are idempotents  then 
\[
(E,i) \star (F,j) = (F,j) \star (E,i) = 
(ij (E \cup F), ij)
\]
By Proposition \ref{prop-defi-equiv-invcat} we conclude that $ \overline{Sz}(\mathcal{C})$ is an inverse category with respect to the pseudo product. 
    
\end{proof}

A motivation for the introduction of the pseudo product for a groupoid is the ESN Theorem for inductive groupoids, which shows that every inductive groupoid is an inverse semigroup with respect to the pseudo product. We can obtain a similar result by introducing the inner Szendrei expansion, where we fix an object $X$ of the inverse category $\mathcal{C}$ and consider the arrows $s:X\rightarrow X$ in the Szendrei expansion $ (\overline{Sz}(\mathcal{C}), \star)$.

\begin{definition}\label{defi-inner-expasion}
	Fix $X$ an object of the inverse category $\mathcal{C}$, the \textit{inner Szendrei expansion} at $X$ is the set $$\overline{Sz}(\mathcal{C}(X)):= P(\mathcal{C}) \rtimes_{(\varepsilon,\mathfrak{B})} \mathcal{C}(X):=\{ (A,s)\in \overline{Sz}(\mathcal{C}); s: X\rightarrow X    \}.$$
\end{definition}

It is clear that $\overline{Sz}(\mathcal{C}(X))$ is a full inverse subcategory of $\overline{Sz}(\calC)$ for every object $X \in \calC$.

The composition of two arrows  $(A,s)$ and $(B,t)$ may not exist in $\overline{Sz}(\mathcal{C}(X))$, but $st$ always exists and therefore, by Proposition \ref{propo-pseudoproduc-invcat},  $(A,s) \star (B,t)$ is defined for each pair of elements of $\overline{Sz}(\mathcal{C}(X))$. 
The next result then follows from Proposition
\ref{prop.szendrei.pseudo.prod}.

\begin{corollary}
    \label{coro-inners-exp-sgi}
	If $\mathcal{C}$ is  an inverse category and  $X\in\mathcal{C}^{(0)}$, the inner Szendrei expansion $\overline{Sz}(\mathcal{C}(X))$ is an inverse semigroup with respect to the pseudo product.
\end{corollary}

\begin{remark}
We want to shed light on the internal structure of $\overline{Sz}(\mathcal{C})$. Consider an idempotent $(E,e)\in\overline{Sz}(\mathcal{C})$ with $E\subset Costar(X)$ and, by definition,  $e^2=e:X\rightarrow X$. Take an arrow $(A,s)$ such that $id(A,s)=(E,e)=ir(A,s)$, that is,  $(s^\circ A, s^\circ s) =(E,e)=(A,ss^\circ)$, and of course  $s:X\rightarrow X$. With this computation, we have discovered that $(A,s)=(E,s)$ satisfies $s^\circ s=e=ss^\circ$, and it is an element of $\overline{Sz}(\mathcal{C}(X))$. 
The conclusion is that the set  $ \overline{Sz}(\mathcal{C})((E,e))$ of such arrows is a group inside the inner expansion $\overline{Sz}(\mathcal{C}(X))$.    
\end{remark}

The following result establishes relations between the outer and inner Szendrei expansions.

\begin{theorem}\label{propo-all-Szendrei-exp}
	Let $\mathcal{C}$ be an inverse category. Then 
	
	\begin{enumerate}[(i)]
		\item $Sz(\mathcal{C})$ is an ordered inverse subcategory of the ordered inverse category $\overline{Sz}(\mathcal{C})$; 
		
	\end{enumerate} 
	
	Furthermore, for each $X\in\mathcal{C}^{(0)}$
	
	\begin{enumerate}
		
		\item [(ii)] ${Sz}(\mathcal{C}(X))$ is an inverse monoid and a subset of the inverse semigroup $\overline{Sz}(\mathcal{C}(X))$;

	\end{enumerate}
\end{theorem}

\begin{proof}
	We start by pointing out that, as $P_\circ(\mathcal{C})$ is a subset of $P(\mathcal{C})$, it becomes a poset with the induced order. Also, the data of Definition \ref{defi-invcat-semidirect-prodc}, grant the inverse categorical structure of the two expansions on item (i). Combining both facts and the arguments used to prove Lemma \ref{lemma-szendrei-exp-is-orderedcat}, we have the claims of the first two items.
	
	For item (ii), the claim that ${Sz}(\mathcal{C}(X))$ is an inverse semigroup has the same proof as Corollary \ref{coro-inners-exp-sgi}. It only remains to be proved that it is also an inverse monoid. 
	
	Let $(A,s)\in Sz(\mathcal{C}(X))$, by definition $A\subset Costar(X), \: A\subset \mathscr{R}_{ses^\circ},\: s:X\rightarrow X  $ and $A\ni ses^\circ, se$. Since $s\in \mathcal{C}(X,X)$, it is possible to compute $s1_X=s=1_Xs$. Moreover, clearly $(\{1_X\},1_X)\in Sz(\mathcal{C}(X))$. So there exists $(A,s)\star(\{1_X\},1_X)$ and $(\{1_X\},1_X)\star(A,s)$, and both are equal to $(A,s)$.	
	Since the strict partial case is analogous, we have finished the proof.	
\end{proof}

\subsection{Strict Bernoulli actions}

While the general Bernoulli actions defined in the previous section allow for a broad choice of indexing maps, the structural properties of inverse categories suggest a natural preference for indices derived from the range idempotents of the subsets. By specializing the indexing maps to specifically track the range idempotent of a set we obtain a more rigid framework that directly reflects the internal partial order of the category $\mathcal{C}$. These are referred to as strict Bernoulli actions, following the terminology and notations by O'Carroll's \cite{ocarroll-strongI}, and Khrypchenko \cite{mykola-partial}, strict inverse semigroups \cite{ocarroll-strongI}.

Given an inverse category $\mathcal{C}$ we define the sets
$$\mathcal{C} {_d\overline{\times} _\varepsilon}P(\mathcal{C}):=\{(A,s)\in \mathcal{C}\times P(\mathcal{C}); o\varepsilon(A)=od(s), \: i\varepsilon(A)=id(s)\},$$  $$\mathcal{C} {_d\overline{\times} _\epsilon}P_\circ(\mathcal{C}):=\{(A,s)\in \mathcal{C}\times P_\circ(\mathcal{C}); o\epsilon(A)=od(s), \: i\epsilon(A)=id(s)\}.$$

Next, we present the \textit{strict Bernoulli actions}.

\begin{definition} Let $\mathcal{C} $ be an inverse category. 
    \begin{enumerate}[(I)]
        \item The \textit{strict global fibred Bernoulli action} is given by the pair $(\varepsilon, \mathfrak{sB})$, where 
        \begin{itemize}
		\item  $\varepsilon: P(\mathcal{C})\rightarrow \mathcal{C}^{(0)}\times RId(\mathcal{C})$, with  $P(\mathcal{C})\mapsto \varepsilon(A)=(o\varepsilon(A),i\varepsilon(A))=(X, e)$;
		\item $\mathfrak{sB}:\mathcal{C} {_d\overline{\times} _\varepsilon}P(\mathcal{C})\rightarrow P(\mathcal{C}) $ with $(s,A)\mapsto \mathfrak{sB}(s,A)=\mathfrak{sB}_s(A):= sA$.
        
	\end{itemize}

Also, for each $s\in\mathcal{C}$ the \textit{strict global Bernoulli action by symmetries} is given by $$\mathfrak{sB}_s: \left\{A\in P(\mathcal{C}); { \begin{cases} o\varepsilon(A)=od(s),\\ i\varepsilon(A)= id(s) \end{cases}} \right\}\rightarrow \left\{B=sA\in P(\mathcal{C}); {\begin{cases}  o\varepsilon(B)=or(s),\\ i\varepsilon(B)= ir(s)  \end{cases}}  \right\}.$$

\item The \textit{strict partial fibred Bernoulli action} is defined by $(\epsilon, \mathfrak{sb})$, where
	\begin{itemize}
		\item  $\epsilon: P_\circ(\mathcal{C})\rightarrow \mathcal{C}^{(0)}\times RId(\mathcal{C})$, with  $P(\mathcal{C})\mapsto \epsilon(A)=(o\epsilon(A),i\epsilon(A))=(X, e)$;
		\item $\mathfrak{sb}:\mathcal{C} {_d\overline{\times} _\epsilon}P_\circ(\mathcal{C})\rightarrow P_\circ(\mathcal{C}) $ with $(s,A)\mapsto \mathfrak{sb}(s,A)=\mathfrak{sb}_s(A):= sA$.
	\end{itemize}

    Moreover, for each $s\in\mathcal{C}$ the \textit{strict partial Bernoulli action by symmetries} is given by   $$\mathfrak{sb}_s: \left\{A\in P(\mathcal{C}); {\begin{cases} o\varepsilon(A)=od(s),\\ i\varepsilon(A)= id(s) \\ A\ni s^\circ, s^\circ  s\end{cases}} \right\}\rightarrow \left\{B=sA\in P(\mathcal{C}); {\begin{cases}  o\varepsilon(B)=or(s),\\ i\varepsilon(B)= ir(s) \\ B\ni ss^\circ,s \end{cases}}  \right\}.$$
    
    \end{enumerate}
\end{definition}

The strict global Szendrei expansion induces the inverse categories

		 $$\overline{Sz}(\mathcal{C})_m = P(\mathcal{C}) \overline{\rtimes}_{(\varepsilon,\mathfrak{sB})} \mathcal{C} \text{ and } \overline{Sz}(\mathcal{C}(-))_m = P(\mathcal{C}) \overline{\rtimes}_{(\varepsilon,\mathfrak{sB})} \mathcal{C}(-).$$

Similarly, the strict partial Szendrei expansion defines the inverse categories
	
	$${Sz}(\mathcal{C})_m = P_\circ(\mathcal{C}) \overline{\rtimes}_{(\epsilon,\mathfrak{sb})} \mathcal{C} \text{ and } {Sz}(\mathcal{C}(-))_m = P_\circ(\mathcal{C}) \overline{\rtimes}_{(\epsilon,\mathfrak{sb})} \mathcal{C}(-).$$

The same conclusions of Theorem \ref{propo-all-Szendrei-exp} holds for the strict expansions.

\begin{theorem}\label{strict-expansions}
    Let $\mathcal{C}$ be an inverse category. Then 
	
	\begin{enumerate}[(i)]
		
	\item $Sz(\mathcal{C})_m$ is an ordered inverse subcategory of the ordered inverse category $\overline{Sz}(\mathcal{C})_m$.
	\end{enumerate} 
	
	Furthermore, for each $X\in\mathcal{C}^{(0)}$
	
	\begin{enumerate}
		
	\item [(ii)] ${Sz}(\mathcal{C}(X))_m$ is a inverse monoid and a subset of the inverse semigroup $\overline{Sz}(\mathcal{C}(X))_m$.
		
	\end{enumerate}
\end{theorem}


\subsection{Enlargements} Our purpose in this section is to develop a notion of enlargement for inverse categories so that it includes the previous definitions of enlargement for ordered groupoids and for inverse semigroups \cite{lawsonlivro}.

\begin{definition}\label{defi-invcat-enalrgements}
	Let $\mathcal{C}$ be an ordered inverse subcategory of the ordered inverse category $\mathcal{D}$. We say that $\mathcal{D}$ is an \textit{enlargement} of $\mathcal{C}$, if
	
	\begin{enumerate}[(I)]
		\item for each $X\in\mathcal{C}^{(0)}$ the set $\ridx{C}{X}$ is an order ideal of $\ridx{D}{X}$;
		
		\item let $X,Y\in\mathcal{C}^{(0)}$, $e \in  \ridx{C}{X}$ and $f \in \ridx{C}{Y}$: if $(s:X\rightarrow Y) \in \mathcal{D}$, and $se=s=fs$, then we have that $s\in\mathcal{C}$;
		
		\item suppose $Y\in\mathcal{D}^{(0)}$ and $f \in \ridx{D}{Y}$: there exists $X\in\mathcal{C}^{(0)}$, $e \in \ridx{C}{X}$ and $s\in\mathcal{D}$ with $s:X\rightarrow Y$ satisfying $s^\circ s=e$ and $ss^\circ=f$. 
	\end{enumerate}
	
	Notation: $\mathcal{C}\subseteq_E \mathcal{D}$.
\end{definition}

\begin{remark}
If $\calC$ and $\calD$ are ordered by the natural order then (I) is a consequence of (II) in Definition \ref{defi-invcat-enalrgements}. In fact, let $X\in\mathcal{C}^{(0)}$, let $e \in \ridx{C}{X}$ and let $f \in \ridx{D}{X}$ be an idempotent such that $e \leq f$, that is, $ef = f = fe$. Then it follows immediately by (II) that $f \in \ridx{C}{X} $.
\end{remark}

The next result will shed light on some categorical aspects of enlargements.

\begin{proposition}\label{propo-invcat-enlarg-implies-cauchyenlarg}
	Suppose $\mathcal{C}$ and $\mathcal{D}$ are inverse categories satisfying $\mathcal{C}\subseteq_E\mathcal{D}$. Then the inclusion functor of their Cauchy completions, $\widehat{\mathcal{C}}\hookrightarrow\widehat{\mathcal{D}}$, is an equivalence.
\end{proposition}

\begin{proof}
By construction, since $\calC$ is a subcategory of $\calD$ it follows that $\widehat{\mathcal{C}}$ is a subcategory of $\widehat{\mathcal{D}}$. We shall prove that the inclusion functor $\inc : \widehat{\mathcal{C}}\rightarrow\widehat{\mathcal{D}}$ is fully faithful and essentially surjective on objects.

It obviously is faithful.
Let $(X,e), (Y,f) \in \widehat{\calC}^{(0)}$ and let $(e,s,f): (X,e) \to (Y,f)$ be a morphism in  $\widehat{\mathcal{D}}$. By definition, $s : X \to Y$ in $\calD$ and  $se=s=fs$. It then follows from  Definition \ref{defi-invcat-enalrgements}  that $s \in \calC$, and therefore $(e,s,f) \in \widehat{\mathcal{C}}$, which proves that $\inc :\widehat{\mathcal{C}}\rightarrow\widehat{\mathcal{D}}$ is a full functor. 

Finally, let $(Y,f)$ be an object of $\widehat{\mathcal{D}}^{(0)}$, that is, $Y \in \calD^{(0)}$ and   $f \in \ridx{D}{Y}$. By Definition \ref{defi-invcat-enalrgements} there exist $X\in\calC$,  $e \in \ridx{C}{X}$ and $s:X\rightarrow Y$ satisfying $e=s^\circ s$ and $f=ss^\circ$.
With this data, we define $(s^\circ s, s, ss^\circ): (X, s^\circ s)\rightarrow (Y,ss^\circ) = (Y,f)$. This arrow is an isomorphism in  $\widehat{\mathcal{D}}$ with inverse $(ss^\circ , s^\circ, s^\circ s): (Y, ss^\circ)\rightarrow (X, s^\circ s)$, since $(ss^\circ , s^\circ, s^\circ s)(s^\circ s, s, ss^\circ) = 1_{(X,s^\circ s)}$		and $(s^\circ s, s, ss^\circ) (ss^\circ , s^\circ, s^\circ s) =1_{(Y,ss^\circ)}$, and this proves that $\inc$ is essentially surjective. 	
	
	Hence, the Cauchy completions of the pair $(\calC, \calD)$ are equivalent categories. \end{proof}

Naturally, one may ask if the  Szendrei expansions may be organized in pairs such that one is an enlargement of the other, that is: Is $\overline{Sz}(\mathcal{C})$ an enlargement of $Sz({\mathcal{C}})$? Is $\overline{Sz}(\mathcal{C})_m$ an enlargement of $Sz(\mathcal{C})_m$? 
We will investigate these questions following the main ideas of Lawson's book (\cite{lawsonlivro},  Chapter 8, Theorem 4).

We begin with the pair $ \overline{Sz}(\mathcal{C})$ and $Sz({\mathcal{C}})$. 	By definition, $Sz(\mathcal{C})$ is an inverse subcategory of  $\overline{Sz}(\mathcal{C})$.
	
(I)  Let $(E,e)\in\overline{Sz}(\mathcal{C})$ and $(F,f)\in  Sz(\mathcal{C})$ be idempotent arrows satisfying $(E,e)\leqslant (F,f)$. We will show that $(E,e)\in Sz(\mathcal{C})$.

 	Note that, since $(E,e)$ is an idempotent arrow,  in order to conclude that $(E,e)\in Sz(\mathcal{C})$ it is enough to verify that $i\varepsilon(E) \in E$. Since $(E,e)\leqslant (F,f)$, we have that $i\varepsilon(E)F\subseteq E$, and since $i\varepsilon(F)\in F$ we obtain that $E\ni i\varepsilon(E)i\varepsilon(F)=i\varepsilon(E).$ Hence (I) holds.

(II) Let $(E,i), (F,j)$ be idempotents in $Sz(\calC)$ and let $(A,s):(E,1_X)\rightarrow (F,1_Y) $ be an arrow in $ \overline{Sz}(\mathcal{C})$ such that $(A,s)(E,i)=(A,s)=(F,j)(A,s).$ Let us fix some notation: 
	$E\subset Costar(X),\:  E\subset \mathscr{R}_e,\: e,i: X\rightarrow X,\:  e\leqslant i,\:  e\in E$,  and  $ F\subset Costar(Y),\: F\subset \mathscr{R}_f,\:  f,j: Y\rightarrow Y,\: f\leqslant j,\: f\in F.$ We can infer a few facts:

	\begin{itemize}
		\item as $(E,i)=id(A,s)=(s^\circ A, s^\circ s) $, then $E=s^\circ A, \: X=od(s) \text{  and  } i=s^\circ s = id(s);$
		
		\item similarly, since $(F,j)=ir(A,s)=(A,ss^\circ)$, we can see $F=A, \: Y= or(s) \text{  and  } j=ss^\circ =ir(s);$
	\end{itemize}

	These facts above imply that $$s \in \calC(X,Y), \: si=s=js,\: A\subset \mathscr{R}_f, \: A\ni f=fj,\: s^\circ A\subset \mathscr{R}_e, \: s^\circ A\ni e=ei   .$$
 Now  $(A,s)$ is an arrow in $Sz{(\mathcal{C})}$ if and only if $i \varepsilon(A), i \varepsilon(A) s \in A$, that is, $f, fs \in A$. 
 
 We have already checked that $f \in A$. In order to verify that $fs \in A$, we note that if $a \in A$ then $a^\circ a = f$, hence 
$$s^\circ f s = s^\circ a^\circ a s = i \varepsilon (s^\circ A) = i \varepsilon (E) = e \in E = s^\circ A; $$ since $A = f A = j f A = ss^\circ f A  =  s(s^\circ A) $ it follows  that $$A \ni s(s^\circ f s) = f ss^\circ s = fs $$
and therefore $(A,s) \in Sz{(\mathcal{C})}$. This shows that the second axiom of an enlargement of inverse categories holds for the extension $Sz{(\mathcal{C})} \subset \overline{Sz}{(\mathcal{C})}$.
 	
 (III)	Consider the idempotent $(F,j)\in\overline{Sz}(\mathcal{C})$, with $F\subset Costar(Y), \: F\subset \mathscr{R}_f \text{  and  } f\leqslant j.$  From Lemma \ref{lemma-global-partial-Bernoulli-proof}, $\mathcal{C}\cdot P_\circ (\mathcal{C})=P(\mathcal{C})$, so there exists an arrow $(s: X\rightarrow U)\in \mathcal{C}$ and an element $A\in P_\circ(\mathcal{C})$, such that $sA$ is defined and $sA=F.$ Notice that $Y=o\varepsilon(F)=o\varepsilon(sA)=or(s)=U$, hence $s\in \mathcal{C}(X,Y)$. Moreover $sA=F\subset \mathscr{R}_f \implies sA\subset \mathscr{R}_f.$ If $A\subset \mathscr{R}_p$, then $sA\subset \mathscr{R}_{sps^\circ}$ and $A\ni p$. We conclude that 
 $f=sps^\circ$. Note also that $sA$ is defined if and only if $A \in dom(\mathfrak{B}_s)$, that is, $X=od(s)=o\varepsilon(A)$ and $p=i\varepsilon(A)\leqslant id(s)=s^\circ s$.
	
	Another conclusion is that $p: X\rightarrow X$, because $or(p)=or(s^\circ s p)=or(s^\circ )=X \text{  and  } p^2=p.$

	Consider the pair $(F,sp)$.  It has the following the properties: 
	\begin{itemize}
		\item $(F,sp)$ is an arrow of $\overline{Sz}(\mathcal{C})$, due to the fact that $ or(sp)=or(s)=Y=o\varepsilon(F)$ and $i\varepsilon(F)=f=sps^\circ \leqslant ir(sp);	$
		
		\item its inner source is $(A,p)$, because $ (sp)^\circ F=ps^\circ F=ps^\circ (sA)=pA=A$, and  $ od(sp)=od(p)=X$, and  $id(sp)=(sp)^\circ (sp)=ps^\circ sp=p;$
		
		\item also, its inner target is $(F,sps^\circ)$, since $ or(sp)=or(s)=Y$, and $ir(sp)= (sp)(sp)^\circ =sps^\circ.$
	\end{itemize}
	
	Summarizing, $(F,sp): A\rightarrow F$ is an arrow in $\overline{Sz}(\mathcal{C})$ such that $(F,sp)^\circ (F,sp)=(A,p) \text{  and  } (F,sp)(F,sp)^\circ =(F,f).$
	
	However, to conclude that the third axiom holds we need the equality $ (F,sp)(F,sp)^\circ= (F,j)$, that is, $ f=ir(sp)=j$, and we have only $f\leqslant j$.

\underline{Conclusion}: the pair $(Sz({\mathcal{C}}), \overline{Sz}(\mathcal{C}))$ satisfies the axioms (I) and (II), but the previous computations do not ensure that (III) holds. 

\

However, observe that they do show that (III) holds whenever the idempotent $(F,j)$ satisfies $j = i\varepsilon (F)$. So we have the next theorem.

\begin{theorem}\label{theo-restrict-exp-enlarg}
	The strict global expansion $\overline{Sz}(\mathcal{C})_m$ is an enlargement of the strict partial expansion $Sz(\mathcal{C})_m$. Moreover, the same relation holds for $\overline{Sz}(\mathcal{C}(X))_m$  and $Sz(\mathcal{C}(X))_m$.
\end{theorem}

\section{Convolution algebras and representation}\label{section-algebras}

\subsection{Morita equivalence of convolution algebras} The last theoretical tool that we need is to define Morita equivalences for categories and to understand how this carries over to their convolution algebras. We will follow the definition of Borceux \cite{borceux2} of Morita equivalence.

\begin{definition}[\cite{borceux-handbook-I}]
\label{defi-Morita-cats}
	Two small categories $\mathcal{C}$ and $\mathcal{D}$ are called \textit{Morita equivalent} if their Cauchy completions $\widehat{\mathcal{C}}$ and $\widehat{\mathcal{D}}$ are equivalent. 	Notation: $\mathcal{C}\simeq_{M}\mathcal{D}$.
\end{definition}

Next we state a proposition of Mitchell \cite{mitchell1972rings}, also found in Xu \cite{xu-representations}, which connects Morita equivalent categories and Morita equivalent convolution algebras. 

\begin{proposition}[\cite{mitchell1972rings}]\label{propo-Morita-conv-alg}
	Let  $\mathcal{C}$ and $\mathcal{D}$ be two small categories, and $\mathbb{K}$ be a commutative unital ring. If $\mathcal{C}^{(0)}$ and $\mathcal{D}^{(0)}$ are finite and $\mathcal{C}$ and $\mathcal{D}$ are Morita equivalent as categories then the convolution category algebras $\mathbb{K}\mathcal{C}$ and $\mathbb{K}\mathcal{D}$ are Morita equivalent as algebras.
\end{proposition}

We can provide a sufficient condition for the Morita equivalence of two convolution algebras, using the concept of enlargements for inverse categories. Let us explain it better: we showed earlier, in  Proposition \ref{propo-invcat-enlarg-implies-cauchyenlarg}, that if an inverse category is an enlargement of another inverse category then their Cauchy completions are equivalent. In particular, the associated restriction groupoids must be Morita equivalent. 

We can state the following theorem.

\begin{theorem}\label{theo-enlarg-imp-morita-alg}
	Let $\mathcal{C}$ and $\mathcal{D}$ be two finite inverse categories and let $\mathbb{K}$ be a commutative unital ring.  If $\mathcal{D}$ is an enlargement of $\mathcal{C}$, then the convolution algebras $\mathbb{K}\mathcal{C}$ and $\mathbb{K}\mathcal{D}$ are Morita equivalent.
\end{theorem}

\begin{proof}
	By Proposition \ref{prop-equiv-cauchy-inv}, if $\mathcal{C}$ and $\mathcal{D}$ are categories then they are  inverse categories if and only if their Cauchy completions $\widehat{\mathcal{C}} \text{  and  } \widehat{\mathcal{D}}$   are inverse categories.
	
	Due to Theorem \ref{theo-Linkcelmann-iso} and the fact that the Cauchy completion is equivalent to its own Cauchy completion (cf. Proposition \ref{propo-cauchy-compl-properties} ), we have that  $\mathbb{K}\widehat{\mathcal{C}}\simeq\mathbb{K}\mathcal{G}_{\mathcal{C} } \text{  and  } \mathbb{K}\widehat{\mathcal{D}}\simeq\mathbb{K}\mathcal{G}_{\mathcal{D} },$ so $\mathcal{C}\simeq_{M}\mathcal{D}$ implies that  $\widehat{\mathcal{C}}\simeq_{M}\widehat{\mathcal{D}}$  therefore $\mathbb{K}{\mathcal{C}}\simeq_{M}\mathbb{K}{\mathcal{D}}$.
\end{proof}

Finally, we can apply it to our case. Nevertheless, first, we will denominate the algebras.

Given a commutative and unital ring $\mathbb{K}$, each Szendrei expansion will give origin to an algebra, which we  now define:
the strict global outer and inner algebras $\mathbb{K}\overline{Sz}(\mathcal{C})_m = \mathbb{K}_{sglob}\mathcal{C}$ and $\mathbb{K}\overline{Sz}(\mathcal{C}(-))_m = \mathbb{K}_{sglob}\mathcal{C}(-)$;
  the strict partial outer and inner algebras  $\mathbb{K}{Sz}(\mathcal{C})_m =\mathbb{K}_{spar}\mathcal{C}$ and $\mathbb{K}{Sz}(\mathcal{C}(-))_m =\mathbb{K}_{spar}\mathcal{C}(-)$.

\begin{corollary}\label{coro-szendrei-alge-morita}
	Given a commutative and unital ring $\mathbb{K}$ and a finite inverse category $\mathcal{C}$, the algebras $\mathbb{K}_{sglob}\mathcal{C}$ and  $\mathbb{K}_{spar}\mathcal{C}$  are Morita equivalent. Furthermore, the algebras $\mathbb{K}_{sglob}\mathcal{C}(X)$ and  $\mathbb{K}_{spar}\mathcal{C}(X)$  are Morita equivalent for each object $X$ in $\mathcal{C}$.
 \end{corollary}
\begin{proof}
	The conclusion follows from Theorem \ref{theo-restrict-exp-enlarg} and Theorem \ref{theo-enlarg-imp-morita-alg}.
\end{proof}

\begin{remark}
    Theorem \ref{theo-enlarg-imp-morita-alg} is related to an analogous result for inverse semigroups. Consider, once more, the  inverse semigroup  $S(\calC)$ constructed by adding a zero to $\calC$ as in Remark \ref{remark:Linckelman-iso-Steinberg}. If $\calC$ is an enlargement of $\calD$ then $S(\calC)$ is an enlargement of $S(\calD)$ in the sense of \cite{lawsonlivro}, and then, by   \cite{steinberg2011strong}, the algebras $\K S(\calC)$ and $\K S(\calD)$ are Morita equivalent. 
\end{remark}


\subsection{Kan extension of representations}
To deal with  representation of inverse categories from a categorical point of view to use Kan extensions (cf. Mac Lane \cite{maclane-categories} and Riehl \cite{riehl-category-context}). This tool allows us to deal with infinite inverse categories, as we show in the next theorem.

\begin{definition}[\cite{xu-representations}]\label{defi-cat-rep}
	Let $\mathcal{C}$ be a category and $\mathbb{K}$ a commutative ring, a \textit{representation} of $\mathcal{C}$ over $\mathbb{K}$ is a covariant functor $\digamma: \mathcal{C}\rightarrow Mod(\mathbb{K})$.	
\end{definition}

The culminating theorem in this section presents the utilization of Kan extensions in precisely characterizing representations of Szendrei expansions of non finite inverse categories.

\begin{theorem}\label{theo-rep-inv-cat-exp}
	Let 
$\mathcal{C}$ be an inverse category and $\mathbb{K}$ be a unital, commutative ring. 
 The representations of $\widehat{\overline{Sz}(\mathcal{C})_m}$ over $\mathbb{K}$ are Kan extensions of the representations of $\widehat{{Sz}(\mathcal{C})_m}$ over $\mathbb{K}$. 
\end{theorem}

\begin{proof}
	From Theorem \ref{theo-restrict-exp-enlarg}  $Sz({\mathcal{C}})_m\subseteq_E \overline{Sz}(\mathcal{C})_m$ implies that $inc:\widehat{Sz({\mathcal{C}})_m}\hookrightarrow \widehat{\overline{Sz}(\mathcal{C})_m}$ is an equivalence.  Since $\mathcal{C}$ is small by definition, $Mod(\mathbb{K})$ is a bicomplete category and the Szendrei expansion is an inverse category, given any functor $\digamma : \widehat{Sz(\mathcal{C})_m}\rightarrow Mod(\mathbb{K})$, there exist its right and left Kan extensions, which are naturally isomorphic (see \cite[Chapter 3]{borceux-handbook-I} and \cite{maclane-categories}).
 	Hence, we have that the representations of $\widehat{\overline{Sz}(\mathcal{C})_m}$ are Kan extensions of the representations of $\widehat{{Sz}(\mathcal{C})_m}$.
\end{proof}

\section{Comparison to previous expansions}\label{section-previous-cases} As mentioned in the introduction, the literature of partial actions presents expansions of the following structures: ordered groupoids, from Gilbert's work \cite{gilbert-pdgexp}; inverse semigroups, from Lawson-Margolis-Steinberg \cite{lawsonmarstein-expansions} that are isomorphic to the Prefix expansion from Buss-Exel \cite{bussexel-expa}; groups, from Exel \cite{ruy98} that are isomorphic to the Birget-Rhodes expansion exposed in Kellendonk-Lawson \cite{kellendonk-lawson}.

Groupoids, inverse semigroups and groups are all particular cases of inverse categories. In this manner, it is possible to obtain a reinterpretation of each one of those expansions using the approach that we developed, that is, via Bernoulli actions and their associated semidirect products.

\subsection{Expansion of ordered groupoids} Let $\mathcal{G}$ be an ordered groupoid, which is an ordered inverse category, as in Definition \ref{defi-ordered-invcat}, where each idempotent arrow is an identity of an object, and inner and outer domain and range maps coincide. 

First, note that applying the constructions of Section \ref{section-actions} to $\mathcal{G}$, we have the set $P_{e}(\mathcal{G}):= \{A\subset \mathcal{G}; \mid A\mid < \infty,\:   r(a)=e, \forall a\in A  \}.$ Then we can define the Bernoulli fibred actions $(\varepsilon,\mathfrak{B}):\mathcal{G}\curvearrowright P(\mathcal{G})$ and $(\epsilon,\mathfrak{b}):\mathcal{G}\curvearrowright_p P_\circ(\mathcal{G})$. Since every idempotent is an identity,  the strict  actions coincide with the non strict actions.  In this manner the Birget-Rhodes expansion $\mathcal{G}^{BR}$, defined by Gilbert \cite{gilbert-pdgexp},  can be rewritten as the semidirect product $\mathcal{G}^{BR}=P_\circ(\mathcal{G})\rtimes_{(\epsilon,\mathfrak{sb})} \mathcal{G}$. Moreover, we can combine our theory with the work of Miller \cite{miller-thesis-structure} and exhibit an enlargement of the groupoid $\mathcal{G}^{BR}$ given as the semidirect product induced by $(\varepsilon,\mathfrak{B})$; using strategies from Clark-Sims \cite{clark-sims-equivgpdSteinbalgebras} we can prove that their groupoid algebras are Morita equivalent.

\subsection{Expansion of inverse semigroups} An inverse semigroup $S$ can be seen as an inverse category with only one object (and possibly without the identity morphism).

In this manner, following Section \ref{section-actions}, we can define the set $P_e(S)= \{A\subset  S; \mid A\mid < \infty,\: a^*a=e , \: \forall a\in A \}$, for each $e\in S$, and the Bernoulli actions by symmetry $\mathfrak{sB}:S\curvearrowright I(P(S))$ and $\mathfrak{sb}\curvearrowright_p I(P_\circ(S) )$. We can combine our theory with the $L$-triple of O'Carroll \cite{ocarroll-strongI} and its identification of partial actions from Khrypchenko \cite{mykola-partial}; then we can conclude that the Prefix expansions from Lawson-Margolis-Steinberg \cite{lawsonmarstein-expansions} and Buss-Exel \cite{bussexel-expa} can be rewritten as the semidirect product $L(S,P(S), P_\circ(S))=P(S)\rtimes_{\mathfrak{sb}}S$. Working in a similar fashion with the global action, we will obtain an inverse semigroup which is the enlargement of $P(S)\rtimes_{\mathfrak{sb}}S$; its groupoid algebras are Morita equivalent, via the algebra isomorphism from Steinberg \cite{steinberg-mobius} and the strategies from Clark-Sims \cite{clark-sims-equivgpdSteinbalgebras}.

\subsection{Expansion of groups} A group $G$ is a particular case of groupoid, and hence of an inverse category. 
Using the same tools from previous expansions we will obtain the sets  $P(G)=\{A\subset G;\: \mid A\mid <\infty  \}$  and $P_e(G)=\{A\subset G; \: \mid A\mid < \infty, \: A\ni e  \}$, where $e$ is the neutral element of $G$. They will lead us to two group actions $\mathfrak{B}:G\curvearrowright P(G)$ and $\mathfrak{b}:G\curvearrowright_p P_e(G)$. Then there is the isomorphism of inverse semigroups $S(G)\simeq P_e(G)\rtimes_{\mathfrak{b}} G$ and we can obtain the global version using the same method with the global action. It is possible to employ the same steps from the inverse semigroup expansion, then we will obtain a Morita context for the partial algebra of $G$, which is $\mathbb{K}_{par}(G)=\mathbb{K}(P_e(G)\rtimes_{\mathfrak{b}}G)$.


\section{Acknowledgments}

This study was financed in part by the 
Coordena\c{c}\~ao de Aperfei\c{c}oamento de Pessoal de N\'ivel Superior– Brasil (CAPES)– Finance Code 001. The first author was partially supported by the National Council for Scientific and Technological Development - CNPq (project 309469/2019-8). This work was partially developed during the second author's post-doctorate at the Graduate Program in Mathematics at the Federal University of Paran\'a, Brazil.

\bibliographystyle{amsplain}
\bibliography{biblio}

@article {ruy98,
    AUTHOR = {Exel, R.},
     TITLE = {Partial actions of groups and actions of inverse semigroups},
   JOURNAL = {Proc. Amer. Math. Soc.},
  FJOURNAL = {Proceedings of the American Mathematical Society},
    VOLUME = {126},
      YEAR = {1998},
    NUMBER = {12},
     PAGES = {3481--3494},
      ISSN = {0002-9939},
   MRCLASS = {46L55 (20M18)},
  MRNUMBER = {1469405},
       DOI = {10.1090/S0002-9939-98-04575-4},
       URL = {https://doi.org/10.1090/S0002-9939-98-04575-4},
}

@article{bagio2010partial,
  title={Partial actions of ordered groupoids on rings},
  author={Bagio, Dirceu and Flores, Daiana and Paques, Antonio},
  journal={Journal of Algebra and its Applications},
  volume={9},
  number={03},
  pages={501--517},
  year={2010},
  publisher={World Scientific}
}

@article{exel1994circle,
  title={Circle actions on {C}*-algebras, partial automorphisms, and a generalized {P}imsner-{V}oiculescu exact sequence},
author={Exel, Ruy},
  journal={Journal of functional analysis},
  volume={122},
  number={2},
  pages={361--401},
  year={1994},
  publisher={Elsevier}
}

@article{birget1984almost,
  title={Almost finite expansions of arbitrary semigroups},
  author={Birget, Jean-Camille and Rhodes, John},
  journal={Journal of Pure and Applied Algebra},
  volume={32},
  number={3},
  pages={239--287},
  year={1984},
  publisher={Elsevier}
}

@article {ruy05,
    AUTHOR = {Dokuchaev, M. and Exel, R.},
     TITLE = {Associativity of crossed products by partial actions,
              enveloping actions and partial representations},
   JOURNAL = {Trans. Amer. Math. Soc.},
  FJOURNAL = {Transactions of the American Mathematical Society},
    VOLUME = {357},
      YEAR = {2005},
    NUMBER = {5},
     PAGES = {1931--1952},
      ISSN = {0002-9947},
   MRCLASS = {16S99 (16S10 16S34 16S35 46L55)},
  MRNUMBER = {2115083},
MRREVIEWER = {Pere Ara},
       DOI = {10.1090/S0002-9947-04-03519-6},
       URL = {https://doi.org/10.1090/S0002-9947-04-03519-6},
}

@book {ruylivro,
    AUTHOR = {Exel, R.},
     TITLE = {Partial dynamical systems, {F}ell bundles and applications},
    SERIES = {Mathematical Surveys and Monographs},
    VOLUME = {224},
 PUBLISHER = {American Mathematical Society, Providence, RI},
      YEAR = {2017},
     PAGES = {vi+321},
      ISBN = {978-1-4704-3785-5},
   MRCLASS = {46L55 (16S35 16S40 37A55 46-02 46L45)},
  MRNUMBER = {3699795},
MRREVIEWER = {Fernando Abadie},
}

@book {lawsonlivro,
    AUTHOR = {Lawson, M.   },
     TITLE = {Inverse semigroups},
      NOTE = {The theory of partial symmetries},
 PUBLISHER = {World Scientific Publishing Co., Inc., River Edge, NJ},
      YEAR = {1998},
     PAGES = {xiv+411},
      ISBN = {981-02-3316-7},
   MRCLASS = {20M18 (18B40)},
  MRNUMBER = {1694900},
MRREVIEWER = {Peter R. Jones},
       DOI = {10.1142/9789812816689},
       URL = {https://doi.org/10.1142/9789812816689},
}

@article {kellendonk-lawson,
    AUTHOR = {Kellendonk, J. and Lawson, M. },
     TITLE = {Partial actions of groups},
   JOURNAL = {Internat. J. Algebra Comput.},
  FJOURNAL = {International Journal of Algebra and Computation},
    VOLUME = {14},
      YEAR = {2004},
    NUMBER = {1},
     PAGES = {87--114},
      ISSN = {0218-1967},
   MRCLASS = {20M18},
  MRNUMBER = {2041539},
MRREVIEWER = {Francis Pastijn},
       DOI = {10.1142/S0218196704001657},
       URL = {https://doi.org/10.1142/S0218196704001657},
}

@article {mishca-ruy-paolo,
    AUTHOR = {Dokuchaev, M. and Exel, R. and Piccione, P.},
     TITLE = {Partial representations and partial group algebras},
   JOURNAL = {J. Algebra},
  FJOURNAL = {Journal of Algebra},
    VOLUME = {226},
      YEAR = {2000},
    NUMBER = {1},
     PAGES = {505--532},
      ISSN = {0021-8693},
   MRCLASS = {16S10 (16S34 20C05)},
  MRNUMBER = {1749902},
MRREVIEWER = {David M. Riley},
       DOI = {10.1006/jabr.1999.8204},
       URL = {https://doi.org/10.1006/jabr.1999.8204},
}

@article {lawson-enlarg,
    AUTHOR = {Lawson, M. },
     TITLE = {Enlargements of regular semigroups},
   JOURNAL = {Proc. Edinburgh Math. Soc. (2)},
  FJOURNAL = {Proceedings of the Edinburgh Mathematical Society. Series II},
    VOLUME = {39},
      YEAR = {1996},
    NUMBER = {3},
     PAGES = {425--460},
      ISSN = {0013-0915},
   MRCLASS = {20M17 (20M18)},
  MRNUMBER = {1417688},
MRREVIEWER = {P. M. Higgins},
       DOI = {10.1017/S001309150002321X},
       URL = {https://doi.org/10.1017/S001309150002321X},
}

@article {steinberg-morita,
    AUTHOR = {Steinberg, B.},
     TITLE = {Strong {M}orita equivalence of inverse semigroups},
   JOURNAL = {Houston J. Math.},
  FJOURNAL = {Houston Journal of Mathematics},
    VOLUME = {37},
      YEAR = {2011},
    NUMBER = {3},
     PAGES = {895--927},
      ISSN = {0362-1588},
   MRCLASS = {46L05 (20M18 22A22)},
  MRNUMBER = {2844456},
}

@article {steinberg-mobius,
    AUTHOR = {Steinberg, B.},
     TITLE = {Mobius functions and semigroup representation theory},
   JOURNAL = {J. Combin. Theory Ser. A},
  FJOURNAL = {Journal of Combinatorial Theory. Series A},
    VOLUME = {113},
      YEAR = {2006},
    NUMBER = {5},
     PAGES = {866--881},
      ISSN = {0097-3165},
   MRCLASS = {20M30},
  MRNUMBER = {2231092},
MRREVIEWER = {Jan Okni\'{n}ski},
       DOI = {10.1016/j.jcta.2005.08.004},
       URL = {https://doi.org/10.1016/j.jcta.2005.08.004},
}

@article {lawsonmarstein-expansions,
    AUTHOR = {Lawson, M.  and Margolis, S. W. and Steinberg,
              B.},
     TITLE = {Expansions of inverse semigroups},
   JOURNAL = {J. Aust. Math. Soc.},
  FJOURNAL = {Journal of the Australian Mathematical Society},
    VOLUME = {80},
      YEAR = {2006},
    NUMBER = {2},
     PAGES = {205--228},
      ISSN = {1446-7887},
   MRCLASS = {20M10 (20M18 20M35)},
  MRNUMBER = {2221438},
MRREVIEWER = {Franco Migliorini},
       DOI = {10.1017/S1446788700013082},
       URL = {https://doi.org/10.1017/S1446788700013082},
}

@article{bussexel-expa,
author = "Buss, A. and Exel, R.",
doi = "10.1215/ijm/1399395828",
fjournal = "Illinois Journal of Mathematics",
journal = "Illinois J. Math.",
number = "4",
pages = "1185--1212",
publisher = "Duke University Press",
title = "Inverse semigroup expansions and their actions on ${C}^{\ast}$-algebras",
url = "https://doi.org/10.1215/ijm/1399395828",
volume = "56",
year = "2012"
}

@article {gilbert-pdgexp,
    AUTHOR = {Gilbert, N. D. },
     TITLE = {Actions and expansions of ordered groupoids},
   JOURNAL = {J. Pure Appl. Algebra},
  FJOURNAL = {Journal of Pure and Applied Algebra},
    VOLUME = {198},
      YEAR = {2005},
    NUMBER = {1-3},
     PAGES = {175--195},
      ISSN = {0022-4049},
   MRCLASS = {20L05 (06F99 18B40 20M18 20M30)},
  MRNUMBER = {2132881},
MRREVIEWER = {G\'{a}bor Braun},
       DOI = {10.1016/j.jpaa.2004.11.006},
       URL = {https://doi.org/10.1016/j.jpaa.2004.11.006},
}

@article {gilbert-graphexp,
    AUTHOR = {Gilbert, N. D.  and Miller, E. C. 
},
     TITLE = {The graph expansion of an ordered groupoid},
   JOURNAL = {Algebra Colloq.},
  FJOURNAL = {Algebra Colloquium},
    VOLUME = {18},
      YEAR = {2011},
    NUMBER = {Special Issue 1},
     PAGES = {827--842},
      ISSN = {1005-3867},
   MRCLASS = {20L05 (20M18)},
  MRNUMBER = {2860365},
MRREVIEWER = {Osman Mucuk},
       DOI = {10.1142/S100538671100071X},
       URL = {https://doi.org/10.1142/S100538671100071X},
}

@incollection {gilbert-pthm,
    AUTHOR = {Gilbert, N. D. },
     TITLE = {A {$P$}-theorem for ordered groupoids},
 BOOKTITLE = {Semigroups and formal languages},
     PAGES = {84--100},
 PUBLISHER = {World Sci. Publ., Hackensack, NJ},
      YEAR = {2007},
   MRCLASS = {06F05 (20N02)},
  MRNUMBER = {2364779},
MRREVIEWER = {M. Mehdi Ebrahimi},
       DOI = {10.1142/9789812708700_0007},
       URL = {https://doi.org/10.1142/9789812708700_0007},
}

@article {cortes-partialactoncat,
    AUTHOR = {Cortes, W. and Ferrero, M. and Marcos, E. N.},
     TITLE = {Partial actions on categories},
   JOURNAL = {Comm. Algebra},
  FJOURNAL = {Communications in Algebra},
    VOLUME = {44},
      YEAR = {2016},
    NUMBER = {7},
     PAGES = {2719--2731},
      ISSN = {0092-7872},
   MRCLASS = {18E05 (16S35)},
  MRNUMBER = {3507148},
MRREVIEWER = {Hamid Rasouli},
       DOI = {10.1080/00927872.2015.1044094},
       URL = {https://doi.org/10.1080/00927872.2015.1044094},
}

@article{clark-sims-equivgpdSteinbalgebras,
title = "Equivalent groupoids have {M}orita equivalent {S}teinberg algebras",
journal = "Journal of Pure and Applied Algebra",
volume = "219",
number = "6",
pages = "2062 - 2075",
year = "2015",
issn = "0022-4049",
doi = "https://doi.org/10.1016/j.jpaa.2014.07.023",
url = "http://www.sciencedirect.com/science/article/pii/S0022404914002187",
author = "Clark, L. O. and Sims, A."}

@article{ocarroll-strongI,
  title={Strongly {E}-reflexive inverse semigroups},
  author={O'Carroll, L.},
  journal={Proceedings of the Edinburgh Mathematical Society},
  volume={20},
  number={4},
  pages={339--354},
  year={1977},
  publisher={Cambridge University Press}
}

@article{mykola-partial,
  title={Partial actions and an embedding theorem for inverse semigroups},
  author={Khrypchenko, M.},
  journal={Periodica Mathematica Hungarica},
  volume={78},
  number={1},
  pages={47--57},
  year={2019},
  publisher={Springer}
}

@book{maclane-categories,
  title={Categories for the working mathematician},
  author={Mac Lane, S.},
  volume={5},
  year={2013},
  publisher={Springer Science \& Business Media}
}

@phdthesis{miller-thesis-structure,
  title={Structure theorems for ordered groupoids},
  author={Miller, E. C.},
  year={2009},
  school={Heriot-Watt University}
}

@phdthesis{hollings-thesis-partial,
  title={Partial actions of semigroups and monoids},
  author={Hollings, C. D.},
  year={2007},
  school={University of York}
}

@article{nystedt-cat-partial-acts,
  title={Partial category actions on sets and topological spaces},
  author={Nystedt, P.},
  journal={Communications in Algebra},
  volume={46},
  number={2},
  pages={671--683},
  year={2018},
  publisher={Taylor \& Francis}
}

@article{linckelmann-inverse,
  title={On inverse categories and transfer in cohomology},
  author={Linckelmann, M.},
  journal={Proceedings of the Edinburgh Mathematical Society},
  volume={56},
  number={1},
  pages={187--210},
  year={2013},
  publisher={Cambridge University Press}
}

@article{cockett-restriction-I,
  title={Restriction categories {I}: categories of partial maps},
  author={Cockett, J. R. B and Lack, S.},
  journal={Theoretical computer science},
  volume={270},
  number={1-2},
  pages={223--259},
  year={2002},
  publisher={Elsevier}
}

@article{xu-representations,
  title={Representations of categories and their applications},
  author={Xu, F.},
  journal={Journal of Algebra},
  volume={317},
  number={1},
  pages={153--183},
  year={2007},
  publisher={Elsevier}
}

@book{borceux-handbook-I,
  title={Handbook of categorical algebra: volume 1, {B}asic category theory},
  author={Borceux, F.},
  volume={1},
  year={1994},
  publisher={Cambridge University Press}
}

@unpublished{dewolf-ehresmann,
  title={The {E}hresmann-{S}chein-{N}ambooripad theorem for inverse categories},
  author={DeWolf, D. and Pronk, D.},
  note={arXiv:1507.08615}
}

@article{steinberg-categories-as-alg-II,
author = {Steinberg, B. and Tilson, B.},
title = {Categories as algebra, {II}},
JOURNAL = {Internat. J. Algebra Comput.},
FJOURNAL = {International Journal of Algebra and Computation},
volume = {13},
number = {06},
pages = {627-703},
year = {2003},
doi = {10.1142/S021819670300150X},
}

@book{maclane-sheaves,
  title={Sheaves in geometry and logic: A first introduction to topos theory},
  author={MacLane, S. and Moerdijk, I.},
  year={2012},
  publisher={Springer Science \& Business Media}
}

@book{riehl-category-context,
  title={{C}ategory {T}heory in {C}ontext},
  author={Riehl, E.},
  isbn={9780486809038},
  lccn={2016498457},
  series={Aurora: Dover Modern Math Originals},
  url={https://books.google.com.br/books?id=Sr09DQAAQBAJ},
  year={2016},
  publisher={Dover Publications}
}

@book{borceux2,
  author = {Borceux, F.},
  publisher = {Cambridge Univ. Press},
  title = {Handbook of Categorical Algebra 2 -- Categories and Structures},
  username = {algebradresden},
  year = 1994
}

@article{petasny2026partial,
  title={Partial semigroupoid actions on sets: RH Petasny, T. Tamusiunas},
  author={Petasny, Rafael Haag and Tamusiunas, Tha{\'\i}sa},
  journal={Semigroup Forum},
  volume={112},
  number={1},
  pages={259--293},
  year={2026}
}

@article{steinberg2011strong,
  title={Strong {Morita} equivalence of inverse semigroups},
  author={Steinberg, Benjamin},
  journal={Houston journal of mathematics},
  volume={37},
  number={3},
  pages={895--927},
  year={2011},
  publisher={University of Houston}
}

@article{mitchell1972rings,
  title={Rings with several objects},
  author={Mitchell, Barry},
  journal={Advances in Mathematics},
  volume={8},
  number={1},
  pages={1--161},
  year={1972},
  publisher={Elsevier},
DOI = {https://doi.org/10.1016/0001-8708(72)90002-3}
}

@article{augusto2026ap,
  title={AP-theorem for inverse semigroupoids through ordered globalizations: FA Tasca et al.},
  author={Augusto Tasca, Felipe and Demeneghi, Paulinho and Mar{\'\i}n, V{\'\i}ctor and Goulart Gomes Velasco, Willian},
  journal={Semigroup Forum},
  volume = {113},
  pages={288--310},
  year={2026}
}

@article{nystedt-pinedo-epsilon,
  title={EPSILON-STRONGLY GROUPOID-GRADED RINGS, THE {P}ICARD INVERSE CATEGORY AND COHOMOLOGY},
  author={Nystedt, P. and {\"O}inert, J. and Pinedo, H.},
  journal={Glasgow Mathematical Journal},
  volume={62},
  number={1},
  pages={233--259},
  year={2020},
  publisher={Cambridge University Press}
}

\end{document}